\documentclass[12pt]{amsart}
\usepackage{amssymb}
\usepackage{verbatim}

\begin{document}

\newtheorem{theorem}{Theorem}    
\newtheorem{proposition}[theorem]{Proposition}
\newtheorem{conjecture}[theorem]{Conjecture}
\def\theconjecture{\unskip}
\newtheorem{corollary}[theorem]{Corollary}
\newtheorem{lemma}[theorem]{Lemma}
\newtheorem{observation}[theorem]{Observation}
\theoremstyle{definition}
\newtheorem{definition}[theorem]{Definition}
\newtheorem{notation}[theorem]{Notation}
\newtheorem{remark}[theorem]{Remark}
\newtheorem{question}[theorem]{Question}

\numberwithin{theorem}{section}
\numberwithin{equation}{section}

\def\earrow{{\mathbf e}}
\def\rarrow{{\mathbf r}}
\def\uarrow{{\mathbf u}}
\def\tpar{{\bf T}}
\def\apar{{\bf A}}

\def\reals{{\mathbb R}}
\def\torus{{\mathbb T}}
\def\heis{{\mathbb H}}
\def\integers{{\mathbb Z}}
\def\complex{{\mathbb C}\/}
\def\naturals{{\mathbb N}\/}
\def\distance{\operatorname{distance}\,}
\def\diff{\operatorname{Diff}\,}
\def\dist{\operatorname{dist}\,}
\def\Span{\operatorname{span}\,}
\def\degree{\operatorname{degree}\,}
\def\kernel{\operatorname{kernel}\,}
\def\dim{\operatorname{dim}\,}
\def\codim{\operatorname{codim}}
\def\trace{\operatorname{trace\,}}
\def\Span{\operatorname{span}\,}
\def\ZZ{ {\mathbb Z} }
\def\p{\partial}
\def\rp{{ ^{-1} }}
\def\Re{\operatorname{Re\,} }
\def\Im{\operatorname{Im\,} }
\def\ov{\overline}
\def\eps{\varepsilon}
\def\lt{L^2}
\def\diver{\operatorname{div}}
\def\curl{\operatorname{curl}}
\def\etta{\eta}
\newcommand{\norm}[1]{ \|  #1 \|}
\def\Span{\operatorname{span}}

\newcommand{\Norm}[1]{ \left\|  #1 \right\| }
\newcommand{\set}[1]{ \left\{ #1 \right\} }
\def\one{\mathbf 1}

\def\scriptf{{\mathcal F}}
\def\scriptg{{\mathcal G}}
\def\scriptm{{\mathcal M}}
\def\scriptb{{\mathcal B}}
\def\scriptc{{\mathcal C}}
\def\scriptt{{\mathcal T}}
\def\scripti{{\mathcal I}}
\def\scripte{{\mathcal E}}
\def\scriptv{{\mathcal V}}
\def\scriptl{{\mathcal L}}
\def\frakv{{\mathfrak V}}

\author{Michael Christ}
\address{
        Michael Christ\\
        Department of Mathematics\\
        University of California \\
        Berkeley, CA 94720-3840, USA}
\email{mchrist@math.berkeley.edu}
\thanks{The author was supported in part by NSF grants DMS-040126 and DMS-0901569.}

\date{April 16, 2009. Revised  June 3, 2011.}

\title 
{On Extremals For a Radon-like Transform}

\begin{abstract}
The operator defined by convolution, in $\reals^d$, with
(affine) surface measure on a paraboloid 
satisfies a dilation-invariant $L^p\to L^q$ inequality,
and enjoys a high-dimensional group of symmetries.
Extremal functions are shown to exist for this inequality.
Moreover, any extremizing sequence is shown to be precompact 
modulo this symmetry group. 
\end{abstract}

\maketitle

\section{Introduction}

If $X,Y$ are Banach spaces, and if $T:X\to Y$ is a bounded
linear operator, 
then by an extremal 
vector for the associated inequality $\norm{Tx}_Y\le C\norm{x}_X$
is meant a nonzero vector $x\in X$
such that $\norm{Tx}_Y = \norm{T}\cdot \norm{x}_X$. 
The following are natural questions:
\begin{enumerate}
\item
What is the norm of $T$?
\item
Does there exist at least one extremal vector $x$?
\item
Are extremal vectors unique, modulo scalar multiplication, or modulo
explicitly known symmetries enjoyed by $T$?
If not, then what is the set of all extremal vectors?
\item If $x$ is a {\em quasiextremal} vector in the sense that
$\norm{Tx}_Y \ge c\norm{T}\norm{x}_X$ for a specified
scalar $c<1$,
then what can be said quantitatively about the structure of $x$?
There are three natural regimes for this question:
(a) for fixed $c$, (b) as $c\to 0$,
and (c) as $c\to 1$. 
\item
What qualititative properties do extremal vectors enjoy?
\item
One can likewise define extremal and quasiextremal pairs
$(x,y)\in X\times Y^*$, where $Y^*$ is the dual of $Y$;
an extremal pair is one which satisfies
$|y(Tx)|= \norm{T}\norm{y}_{Y^*}\norm{x}_X$.
Each of questions (2) through (5) has an analogue for
such pairs.
\end{enumerate}

This paper investigates these questions,
for a specific operator which arises
in Euclidean harmonic analysis.
This operator is modestly fundamental,
as a basic example of a Fourier integral
operator, of a Radon-like transform, and of an
object whose analytic properties are governed by
geometry and combinatorics. Within those frameworks
it occupies a unique niche, attested to by
an associated high-dimensional symmetry group,
connected to the Heisenberg group and metaplectic
representation. 

Regard $\reals^d$ as $\reals^{d-1}\times\reals$
with coordinates $x=(x',x_d)$.
Let $\sigma$ be the singular measure on $\reals^d$,
supported on the paraboloid $\{x: x_d=|x'|^2\}$,
defined by 
$\int f\,d\sigma
= \int_{\reals^{d-1}} f(x',|x'|^2)\,dx'$.
$\sigma$ and surface measure on the paraboloid
$x_d=|x'|^2$ are mutually absolutely continuous, but $\sigma$
enjoys a certain dilation symmetry which surface measure lacks.
See \S\ref{section:affine} for discussion of a general context for
such a measure.

The convolution operator and inequality under discussion are:
\begin{gather} 
\tpar f = f*\sigma
\\
\norm{\tpar f}_{L^{d+1}(\reals^d)}
\le \apar \norm{f}_{L^{(d+1)/d}(\reals^d)}
\label{eq:paraboloid}
\end{gather}
for all $f\in C^0\cap L^\infty\cap L^1(\reals^d)$.
Here $\apar<\infty$ denotes  the optimal constant in the inequality,
which depends only on the dimension $d$.
For some discussion of the fundamental nature of
this inequality, see \cite{quasiextremal}.

In this paper we establish the existence of extremals,
and of extremal pairs, for 
$\tpar$ in all dimensions. 
Moreover,
we obtain uniform quantitative information about the behavior
of extremals, and of $(1-\delta)$-quasiextremals with $\delta$ close to $0$. 
We show that sequences of approximate extremizers are 
precompact, after renormalization by an explicitly described symmetry group.
This complements our earlier work \cite{quasiextremal}
concerning $c$-quasiextremals, in the regime
where $c$ is not close to $1$.
Various natural questions remain open:
We neither identify these extremals  and the optimal constant $\apar$, 
nor address the question of their uniqueness modulo symmetries, 

One of our objectives is
to begin to develop techniques which 
should be useful in an analysis of extremals for
other closely related operators. 
See \cite{christshao} for one such problem.
This paper is the second in a series
treating aspects of the meta-question: If the ratio
$\Phi(f)=\|Tf\|_{q}/\|f\|_p$ is large, then what are the 
properties of $f$? The word ``large'' admits various interpretations.
The initial work \cite{quasiextremal} is a study of those functions
$f$ for which the $\Phi(f)$ is bounded below by some positive constant.
In \cite{christxue}, qualitative properties of arbitrary critical
points of $\Phi$ are studied.
The paper \cite{christradonextremals} demonstrates an equivalence
between the inequality studied here, and a certain inequality for
the Radon transform, and explicitly identifies all extremizers for both.
This equivalence, together with an rearrangement inequality proved in
\cite{christkplane} and special considerations for the subclass of all
radially symmetric functions, could be used to give an alternative proof
of the existence of extremizers for $\Phi$. However, the arguments given
in the present paper, which ultimately rely on qualitative rather than exact
symmetries, are more general and therefore retain some interest.  
One problem which has such qualitative but not exact symmetries is studied in \cite{christshao}.

We are indebted to Ren\'e Quilodr\'an and to Shuanglin Shao for useful advice on the exposition,
and to Terence Tao and Shuanglin Shao for posing related questions to us.

\section{Results}
Define
\[
\apar = \sup_{\norm{f}_{(d+1)/d}=1} \norm{\tpar f}_{d+1}.
\]
\begin{definition}
An extremizer for the inequality \eqref{eq:paraboloid}
is a function $f\in L^{(d+1)/d}(\reals^d)$ which satisfies
$\norm{\tpar f}_{L^{d+1}} = \apar\norm{f}_{L^{(d+1)/d}}\ne 0$.

An extremizing sequence for the inequality \eqref{eq:paraboloid}
is a sequence of nonnegative functions $f_\nu\in L^{(d+1)/d}$ satisfying 
\begin{align}
&\norm{f_\nu}_{L^{(d+1)/d}}\equiv 1
\\
&\norm{\tpar f_\nu}_{L^{d+1}}\to \apar.
\end{align}

For any $\delta\in[0,1)$,
$f\in L^{(d+1)/d}$ is a $(1-\delta)$-quasiextremal for \eqref{eq:paraboloid}
if 
\begin{equation}
\norm{\tpar f}_{L^{d+1}} \ge (1-\delta)\apar\norm{f}_{L^{(d+1)/d}}\ne 0.
\end{equation}
\end{definition}


In Definition~\ref{defn:symmetries} below we will introduce a group
$\scriptg_d$ of diffeomorphisms of $\reals^d$ which are natural symmetries
of our problem. Associated to each $\phi\in\scriptg_d$
is an invertible linear operator $\phi^*$
on $L^{(d+1)/d}(\reals^{d})$ satisfying 
$\norm{\phi^* f}_{(d+1)/d} = 
\norm{f}_{(d+1)/d}$ and 
$\norm{\tpar(\phi^* f)}_{d+1} = \norm{\tpar(f)}_{d+1}$ for all $f\in L^{(d+1)/d}$. 
Thus for any sequence of functions $f_\nu$ and elements
$\phi_\nu\in\scriptg_d$,
$(f_\nu)$ is an extremizing sequence if and only if
$(\phi_\nu^* f_\nu)$ is an extremizing sequence;
extremizing sequences can only be characterized modulo the
action of $\scriptg_d$.

\begin{theorem} \label{thm:main}
\noindent
(i)
There exist extremizers for the inequality \eqref{eq:paraboloid}.

\noindent
(ii)
Let $\{f_\nu\}$ be any extremizing sequence for the inequality \eqref{eq:paraboloid}.
Then there exist  an extremal $f$ for \eqref{eq:paraboloid} 
satisfying $\norm{f}_{(d+1)/d}=1$, 
a subsequence $\{f_{\nu_i}\}$, 
and a sequence of symmetries $\phi_i\in {\scriptg}_d$ such that
$\phi_i^*(f_{\nu_i})\to f$ in $L^{(d+1)/d}$ norm.

\noindent
(iii)
There exist a constant $C_0<\infty$,
a function $\Psi:[0,\infty)\to[0,\infty)$
satisfying $\Psi(t)\ge t^{(d+1)/d}$ for all $t$
and
$\Psi(t)/t^{(d+1)/d}\to\infty$ as $t\to 0^+$
and also as $t\to\infty$,
and a function $\rho:[1,\infty)\to(0,\infty)$
satisfying $\rho(R)\to 0$ as $R\to\infty$,
with the following property.
For any 
extremizing sequence $\{f_\nu\}$
for the inequality \eqref{eq:paraboloid},
there exists a sequence of elements $\phi_\nu\in {\scriptg}_d$ 
such that for all sufficiently large $\nu$,
$\phi_\nu^* f_\nu$ can be decomposed as
$g_\nu+h_\nu$
so that $\norm{h_\nu}_{(d+1)/d}\to 0$,
\begin{equation}\label{concl:higherintegrability}
\int_{\reals^d} \Psi(g_\nu)\le C_0;
\end{equation}
for all $R\ge 1$,
\begin{equation} \label{concl:decay} 
\int_{|x|\ge R}  g_\nu(x)^{(d+1)/d}\,dx\le \rho(R).
\end{equation}
Moreover
for any $R\ge 1$,
for all $\nu\ge\nu(R)$,
$g_\nu$ may be further decomposed as 
$g^\sharp_\nu+ g^\flat_\nu$
where both summands continue to satisfy
\eqref{concl:higherintegrability} and \eqref{concl:decay},
and moreover
\begin{equation} \label{concl:smoothness}
\norm{g^\sharp_\nu}_{C^1}\le R
\text{ and } 
\norm{g^\flat_\nu}_{L^{(d+1)/d}}\le \eta(R)
\end{equation}
where $\eta(R)\to 0$ as $R\to\infty$,
and the function $\eta$ is independent of the extremizing sequence.

\noindent
(iv)
For any nonnegative extremizer $f$ for \eqref{eq:paraboloid}
satisfying $\norm{f}_{(d+1)/d}=1$,
there exists $\phi\in {\scriptg}_d$ such that
$\phi^* f$ satisfies \eqref{concl:higherintegrability},
\eqref{concl:decay}, and \eqref{concl:smoothness}.

\noindent
(v)
Any complex-valued extremizer for \eqref{eq:paraboloid}
agrees almost everywhere with $e^{i\theta}f$ for some
$C^\infty$ nonnegative function $f$ and some constant $\theta\in\reals$.

\noindent 
(vi)
Any nonnegative extremizer for inequality \eqref{eq:paraboloid}
with $\norm{f}_{(d+1)/d}=1$
satisfies the Euler-Lagrange equation
\begin{equation} \label{eulerlagrange}
\tpar^*([\tpar f]^d)=\apar^{d+1} f^{1/d}
\end{equation}
almost everywhere. 
For any compact set $K\subset\reals^d$
there exists $c>0$ such that
$f(x)\ge c$ for almost every $x\in K$.
\end{theorem}
\noindent
In conclusion {\it (vi)}{}, $\tpar^*$ denotes the transpose of $\tpar$.

\begin{remark}
Extremizers of \eqref{eq:paraboloid} do not belong to the Schwartz class.
Indeed, if $f$ is a continuous nonnegative solution of the Euler-Lagrange equation
\eqref{eulerlagrange} which does not vanish identically,
then $\limsup_{|x|\to\infty} |x|^{d/2}f(x)>0$.
For $\tpar f$ is bounded away from zero in some
ball $B$. The equation and nonnegativity together force
$f^{1/d}\ge c \tpar^*(\chi_B)$, and consequently $f(x',x_d)\ge c(1+|x'|)^{-d}$
for all $(x',x_d)$ belonging to a certain tubular neighborhood,
of constant width, of a paraboloid.
In such a region, $|x'|\sim |x|^{1/2}$ as $|x|\to\infty$.
\end{remark}

Quasiextremals for the inequality \eqref{eq:paraboloid} were studied 
in \cite{quasiextremal}. 
These are by definition functions which satisfy $\norm{\tpar f}_{d+1}
\ge \eps \norm{f}_{(d+1)/d}$ for an arbitrary constant $\eps>0$;
$\eps$ need not be close to the optimal constant $\apar$.
The analytic techniques
introduced there were further developed in \cite{betsyquasiextremal}
to treat quasiextremals for the corresponding inequality for $T_{\rm sphere}$.
These techniques form the basis of the present paper.
However, we have structured the exposition to emphasize certain geometric
facets of the subject, especially the symmetry group and family of paraballs introduced
in \S\ref{section:preliminaries}, which were less fully developed
in \cite{quasiextremal}. 

Throughout the paper we assume that the dimension of the
ambient space $\reals^d$ satisfies $d\ge 2$.
Define 
\[p =\frac{d+1}{d} \text{ and }  q= d+1.\]
Since $|T(f)|\le T(|f|)$ and $\norm{|f|}_{p}= \norm{f}_p$,
there exist extremal functions for the inequality in question
if and only if there exist nonnegative extremal functions;
we assume henceforth, without loss of generality and without
further comment, that all functions under discussion are nonnegative.
$c,C,\gamma$ will denote finite positive constants which depend only
on the dimension $d$, and which are permitted to change values from
one occurrence to the next. Typically $c$ will be small, while $C$ will
be large.


\section{Preliminaries} \label{section:preliminaries}
We begin by collecting various facts which will be used in the analysis.

\subsection{Symmetries}

The operator $\tpar$ enjoys a relatively high-dimensional
Lie group of symmetries.
Let $\Theta:\reals^{d+d}\to\reals$ be the function
\[
\Theta(x,y) = x_d-y_d -|x'-y'|^2
\]
and let
$\scripti$ be the incidence manifold
$\scripti=\{(x,y)\in\reals^{d+d}: \Theta(x,y)=0\}$. 
$\diff(\reals^n)$ denotes the group of all $C^\infty$
diffeomorphisms of $\reals^n$.

\begin{definition}\label{defn:symmetries}
$\scriptg_{d,d}$ denotes the set of all
ordered pairs $(\phi,\psi)\in\diff(\reals^d)\times\diff(\reals^d)$
which preserve $\scripti$ in the strong sense
that there exists $0\ne\lambda\in\reals$ such that
\begin{equation} \label{eq:preserveincidence}
\Theta(\phi(x),\psi(y))=\lambda\Theta(x,y)
\text{ for all $(x,y)\in\reals^{d+d}$}.
\end{equation}

$\scriptg_d$ denotes the set of all $\phi\in\diff(\reals^d)$
for which there exists $\psi$ such that $(\phi,\psi)\in\scriptg_{d,d}$.
\end{definition}
\noindent
In particular, $(x,y)\in\scripti\Rightarrow (\phi(x),\psi(y))\in\scripti$.


The following are examples of elements of $\scriptg_{d,d}$:
\begin{itemize}
\item 
$(\phi(x),\psi(y)) = 
(x+v,y+v)$, for any $v\in\reals^d$
\item 
$(\phi(x),\psi(y)) = 
(rx',r^2 x_d; ry',r^2 y_d)$, for any $r\in\reals\setminus\{0\}$
\item
$(\phi(x),\psi(y)) = 
(x'+u,x_d+2u\cdot x'+|u|^2; y',y_d+2u\cdot y')$,
for any $u\in\reals^{d-1}$
\item
$(\phi(x),\psi(y)) = 
(L(x'),x_d +|Lx'|^2-|x'|^2; L^\dagger(y'),y_d-|L^\dagger(y')|^2+|y'|^2)$,
for any invertible linear transformation $L:\reals^{d-1}\to\reals^{d-1}$,
where $L^\dagger$ is the inverse of the transpose of $L$.
\end{itemize}

Denote by $\tau$ the involution $x\mapsto -x$ of $\reals^d$.

\begin{lemma}
(i) The correspondence $(\phi,\psi)\leftrightarrow\phi$ is a bijection
of $\scriptg_{d,d}$ with $\scriptg_d$.

(ii)
If $(\phi,\psi)\in\scriptg_{d,d}$ then
$\tau\circ\psi\circ\tau\in\scriptg_d$.

(iii)
Every $\phi\in\scriptg_d$ is of the form
$(x',x_d)\mapsto (Lx',Ax_d+Q(x'))$
for some invertible affine endomorphism $L$ of $\reals^{d-1}$,
some invertible affine endomorphism $A$ of $\reals^1$, and some polynomial
$Q:\reals^{d-1}\to\reals$ of degree $\le 2$.
The pair $(L,A)$ is unrestricted. $Q$ 
takes the form $Q=q+a$ where $q$ is a homogeneous polynomial
of degree $2$ uniquely determined
by $(L,A)$, while $a$ is an arbitrary affine mapping.
\end{lemma}

\begin{proof}
(ii) is a direct consequence of the definitions.

To establish (iii), write $\phi(x) = (F(x),f(x))\in\reals^{d-1}\times\reals$
and similarly  $\psi(y) = (G(y),g(y))$.
To the equation 
\[f(x)-g(y)-|F(x)-G(y)|^2= \lambda \big(x_d-y_d-|x'-y'|^2 \big)\]
apply any mixed second partial derivative
$\frac{\partial^2}{\partial x_i\partial y_j}$
to conclude that 
\[
\Big\langle
\frac{\partial F(x)}{\partial x_i}\ ,\ \frac{\partial G(y)}{\partial y_j}
\Big\rangle
\text{ is independent of $x,y$}.
\]
Since $\reals^d\owns x\mapsto F(x)\in\reals^{d-1}$ is a submersion
at every $x$, it follows that $\nabla G(y)$ is independent of $y$;
similarly $\nabla F(x)$ is independent of $x$. Thus $F,G$
are affine functions. 

If $i=d$ or $j=d$ then $\partial^2/\partial x_i\partial y_j$
annihilates $x_d-y_d-|x'-y'|^2$,
so it must annihilate $F(x)\cdot G(y)$. It follows again
from the submersion property that $F,G$ are independent 
of $x_d,y_d$.

By comparing terms we see that $F(x)\cdot G(y)\equiv \lambda x'\cdot y'$
plus an affine function of $x',y'$.
Therefore $F,G$ take the form
$F(x) = Ax'+u$ and $G(y)=\lambda By'+v$, where $B$ is the transpose of $A^{-1}$
and $u,v$ are vectors in $\reals^{d-1}$.

Now $f(x)-g(y)$ equals $\lambda x_d-\lambda y_d$ plus a quadratic polynomial in $x',y'$.
This forces $f$ to be a quadratic polynomial in $(x',x_d)$ in which the
coefficient of $x_dx_j$ vanishes for all $1\le j\le d$. Likewise for $g$.
The other parts of the description of $\phi,\psi$ now follow from the equation.

\medskip
In order to prove that $(\phi,\psi)\to\phi$ is a bijection, 
which was conclusion (i), it suffices to
prove that if $\phi$ is the identity, then so must be $\psi$.
We already know that $\psi$ must take the form $\psi(y',y_d) = (\lambda y'+v,g(y))$.
Thus 
\[x_d-g(y)-|x'-y'-v|^2\equiv \lambda\big(x_d-y_d-|x'-y'|^2 \big).\]
Equating coefficients of $x_d$ forces $\lambda=1$.
Therefore
$y_d\equiv g(y)+2v\cdot(x'-y')+|v|^2$.
Since the left-hand side is independent of $x'$, $v$ must vanish, 
leaving $g(y)\equiv y_d$.
\end{proof}

Some further information concerning 
$\scriptg_d$ may be found in \S\ref{onsymmetry} below.

For each $\phi\in {\scriptg}_d$, the Jacobian determinant $J_\phi:\reals^d\to (0,\infty)$
of $\phi$ is a constant function.
Associated to each element $\phi\in {\scriptg}_d$ is the 
mapping $\phi^*:L^{(d+1)/d}(\reals^d)\to L^{(d+1)/d}(\reals^d)$
defined by $\phi^*f(x) = f(\phi(x))J^{d/(d+1)}$.
Then 
\begin{equation}
\begin{split}
\norm{\phi^*f}_{(d+1)/d}=\norm{f}_{(d+1)/d} 
\\
\langle \psi^*g,\tpar(\phi^*f)\rangle = \langle g,\tpar(f)\rangle
\end{split}
\end{equation}
for all $f,g\in L^{(d+1)/d}$.

\subsection{Paraballs}

\begin{definition} \label{defn:paraballs}
Let $r_1,\cdots,r_{d-1},\rho>0$,
let $\earrow=\{e_1,\cdots,e_{d-1}\}$ 
be any orthonormal basis for $\reals^{d-1}$,
let $\rarrow=(r_1,\cdots,r_{d-1})\in (0,\infty)^{d-1}$,
and let $z=(\bar x,\bar x_\star)\in\scripti$,
that is, $z\in\reals^{d}\times\reals^d$
satisfies $(\bar x_\star)_d-\bar x_d = |\bar x'_\star-\bar x'|^2$.
The {\em paraball}
$B=B(z,\earrow,\rarrow,\rho)$ associated to these data
is the set of all  $x\in\reals^d$ satisfying
\begin{gather}
\label{firstfirstcoords}
\sum_{j=1}^{d-1}
r_j^{-2}|\langle x'-\bar x',e_j\rangle|^2< 1,
\\
\label{firstlastcoord}
\big| x_d-(\bar x_\star)_d -|x'-\bar x_\star'|^2\big| <\rho.
\end{gather}
Here $\bar x=(\bar x',\bar x_d)\in\reals^d$
and $\bar x^\star=(\bar x'_\star,(\bar x^\star)_d)\in\reals^d$.

For any $\lambda\ge 1$, the expanded paraball
$\lambda B(z,\earrow,\rarrow,\rho)$
is defined to be the set of all $x\in\reals^d$ satisfying
\begin{gather}
\label{expandedfirstcoords}
\sum_{j=1}^{d-1}
r_j^{-2}|\langle x'-\bar x',e_j\rangle|^2<\lambda^2, 
\\
\label{expandedlastcoord}
\big| x_d-(\bar x_\star)_d -|x'-\bar x_\star'|^2\big| <\lambda \rho.
\end{gather}

The dual paraball
$B_\star=B_\star(z,\earrow,\rarrow,\rho)$ associated to these data
is the set of all $x^\star=(x'_\star,(x_\star)_d)\in\reals^d$ satisfying
\begin{gather}
\label{dualfirstcoords}
\sum_{j=1}^{d-1} 
(r^{\star}_j)^{-2}
|\langle x_\star'-\bar x_\star',e_j\rangle|^2< 1
\\
\label{duallastcoord}
\big| (x_\star)_d - \bar x_d + |x_\star'-\bar x'|^2 \big| <\rho
\end{gather}
where $r_jr^\star_j=\rho$ for every $j\in\{1,2,\cdots,d-1\}$.

$\scriptb(z,\earrow,\rarrow,\rho)$ denotes the ordered
pair $\scriptb=(B(z,\earrow,\rarrow,\rho),B_\star(z,\earrow,\rarrow,\rho))$.
\end{definition}

$B$ uniquely determines $B_\star$, and vice versa.
The point $z=(\bar x,\bar x_\star)$ is an element of $\reals^{d}\times\reals^d$,
not of $\reals^d$. The data $(\bar x,\earrow,\rarrow,\rho)$ do not
suffice to completely determine $B$;
the point $\bar x$ can be regarded as a ``center''
of $B$, but the geometry of $B$ depends also on $\bar x_\star$.

The mapping from data $(z,\earrow,\rarrow,\rho)$ to paraballs
is not a one-to-one correspondence; $O(d-1)$ acts naturally
on the set of all orthonormal bases $\earrow$, and 
if $\earrow,\tilde\earrow$ belong to the
same orbit under this action then $B(z,\earrow,\rarrow,\rho)
=B(z,\tilde\earrow,\rarrow,\rho)$. 

\medskip
Given $\phi\in\scriptg_d$ and a paraball $B$,
$\{x: \phi(x)\in B\}$ is a paraball. 
For $\phi$ can be expressed in the form
$\phi(x',x_d) = v + (L(x'),\lambda x_d+Q(x'))$
where $L\in Gl(d-1)$, $0\ne\lambda\in\reals$,
$v\in\reals^d$, and $Q:\reals^{d-1}\to\reals^1$. 
Then $L^*L$ is a positive definite symmetric matrix, 
hence admits a factorization $L^*L=O^{-1}DO$ with
$O\in O(d-1)$ and $D$ a diagonal matrix with positive diagonal entries.
The condition $|L(x')|^2\le 1$
is equivalent to 
$|\langle D^{1/2}O(x'),D^{1/2}O(x')\rangle|\le 1$,
which can be expressed in terms of $\earrow,\rarrow$
where $\earrow$ is the image of the standard basis
for $\reals^{d-1}$ under $O$,
and $(r_1,\cdots,r_{d-1})$ are the diagonal entries
of $D^{1/2}$.
Thus \eqref{firstfirstcoords} transforms under $\phi$ to another
inequality of the same form. 
This is the main step in establishing
the first conclusion of the next lemma.

\begin{lemma}
There are natural actions of $\scriptg_d$ on the
set of all paraballs, and
of $\scriptg_{d,d}$ on the set of all dual pairs of paraballs.
These actions are transitive. 
\end{lemma}

The simple remainder of the verification is left to the reader.

\subsection{Distance between paraballs}
It will be useful to quantify the notion that two paraballs are far apart. 
An approximate measure $\varrho(B^\sharp,B^\flat)$ of
the discrepancy between two paraballs is defined as follows.

To any paraball 
$B=B(\bar x, \earrow,\rarrow,\rho)$ is associated 
the balanced convex subset $\scriptc\subset \reals^{d-1}$ defined by
\[
\scriptc=\Big\{
v\in\reals^{d-1}: 
\sum_{j=1}^{d-1}
r_j^{-2}|\langle v,e_j\rangle|^2 <1\Big\}\ .
\]
Much of the geometry of $B$ is encoded by $\scriptc$;
$B$ is the set of all $(x',x_d)$ such that $x'-\bar x'\in \scriptc$
and $|x_d-h(x')|<\rho$, for a certain quadratic polynomial $h$ 
specified by $\bar x_\star$. 

In the following definition, $\scriptc^\sharp,\scriptc^\flat$
denote the convex sets associated to $B^\sharp,B^\flat$,
respectively.

\begin{definition} \label{defn:varrho}
Let $B^\sharp=B(z^\sharp,\earrow^\sharp,\rarrow^\sharp,\rho^\sharp)$ and
$B^\flat=B(z^\flat, \earrow^\flat,\rarrow^\flat,\rho^\flat)$ 
be two paraballs,
where $z^\sharp = (\bar x^\sharp,\bar x^\sharp_\star)$,
$z^\flat= (\bar x^\flat,\bar x^\flat_\star)$ are two points in $\scripti$.
Define 
\begin{align*}
\varrho(B^\sharp,B^\flat) 
&= 
\frac{\max(\rho^\sharp,\rho^\flat)}{\min(\rho^\sharp,\rho^\flat)}
\ +\ 
\sup_{v\in\scriptc^\sharp} 
\sum_{j=1}^{d-1} \frac{|\langle v,e^\flat_j\rangle|^2} {(r_j^\flat)^{2}}  
\ +\ 
\sup_{v\in\scriptc^\flat} 
\sum_{j=1}^{d-1} \frac{|\langle v,e^\sharp_j\rangle|^2} {(r_j^\sharp)^{2}}
\\
&\ +\ 
\sum_{j=1}^{d-1} \frac{|\langle \bar x^{\sharp}{}'-\bar x^{\flat}{}',e^\sharp_j\rangle|^2} 
{(r_j^\sharp)^{2}}
\ +\ 
\sum_{j=1}^{d-1} \frac{|\langle \bar x^{\sharp}{}'-\bar x^{\flat}{}',e^\flat_j\rangle|^2} 
{(r_j^\flat)^{2}}  
\\
&\ +\ 
\sum_{j=1}^{d-1} 
\frac{|\langle \bar x_\star^{\sharp}{}'
-\bar x_\star^{\flat}{}',e^\sharp_j\rangle|^2} {(\rho^\sharp/r_j^\sharp )^{2}}
\ +\ 
\sum_{j=1}^{d-1} 
\frac{|\langle \bar x_\star^{\sharp}{}'
-\bar x_\star^{\flat}{}',e^\flat_j\rangle|^2} {(\rho^\flat/r_j^\flat)^2} 
\\
&\ +\ 
\frac{|\bar x_d^\flat-(\bar x_\star^\sharp)_d 
- \big|\bar x^\flat{}'-\bar x_\star^\sharp{}'|\big|}{\rho^\sharp}
\ +\ 
\frac{|\bar x_d^\sharp-(\bar x_\star^\flat)_d 
- \big|\bar x^\sharp{}'-\bar x_\star^\flat{}'|\big|}{\rho^\flat} \ .
\end{align*}
In particular, $\varrho(B^\sharp,B^\flat)\ge 1$  and
$\varrho(B^\sharp,B\flat) = \varrho(B^\flat,B^\sharp)$ for any $B^\sharp,B^\flat$, and
$\varrho(B,B)=1$ for any $B$. 
\end{definition}

$\varrho$ is one of many ways in which the difference between two paraballs
can be quanfified, and other variants would serve our purpose equally well.
Of primary interest here is the situation in which two paraballs differ
markedly, rather than when they nearly coincide.
$\varrho$ is not actually a metric; for instance,
$\varrho(B^\sharp,B^\flat)\ge 1$ for any $B^\sharp,B^\flat$.

The following two invariance properties are direct consequences
of Definitions~\ref{defn:paraballs} and \ref{defn:varrho}.

\begin{lemma}
\label{lemma:dualballsclose}
For any $\phi\in\scriptg_d$
and any two paraballs $B^\sharp,B^\flat$,
\begin{equation}
\varrho\big(\phi(B^\sharp),\phi(B^\flat)\big)
= \varrho (B^\sharp,B^\flat).
\end{equation}

For any two dual pairs of paraballs
$(B^\sharp,B_\star^\sharp)$
and 
$(B^\flat,B_\star^\flat)$,
\begin{equation}
\varrho(B^\sharp,B^\flat) 
=\varrho(B_\star^\sharp,B_\star^\flat).
\end{equation}
\end{lemma}

The next lemma relates the distance between two paraballs
to the relative size of their intersection.

\begin{lemma} \label{lemma:ifparaballsintersect}
There exists a constant $C<\infty$ 
which depends only on the dimension $d$,
such that for any two paraballs $B^\sharp,B^\flat$,
\begin{equation}
\varrho(B^\sharp,B^\flat)\le 
C\left( \frac{\max(|B^\sharp|,|B^\flat|)}{|B^\sharp\cap B^\flat|}\right)^{C}.
\end{equation}
\end{lemma}

There is of course no converse inequality; the right-hand side
becomes infinite whenever $B^\sharp\cap B^\flat=\emptyset$. 

The following elementary fact will be used in the proof:
For any $d\ge 1$ there exist constants $C,c\in(0,\infty)$
such that for any convex set $\scriptc\subset\reals^d$
of positive Lebesgue measure, and any quadratic polynomial
$Q:\reals^d\to\reals$ which does not vanish identically,
for any $\eps>0$
\[
|\{y\in\scriptc: |Q(y)|<\eps \sup_{\scriptc}|Q|\}|
\le C\eps^c|\scriptc|.
\]

\begin{proof}[Proof of Lemma~\ref{lemma:ifparaballsintersect}]
We need to show that a lower bound on $\varrho(B^\sharp,B^\flat)$
implies an upper bound on 
$|B^\sharp\cap B^\flat|/\max(|B^\sharp|,|B^\flat|)$ of inverse
power law type.  We may assume that $\varrho(B^\sharp,B^\flat)$
is large, since otherwise the inequality holds automatically
for large enough $C$.

Recall
the description of $B^\flat$ as the set of all $(x',x_d)$
such that $x'\in \scriptc^\flat$,
and $|x_d-h(x')|<\rho$, for a certain function $h$ 
which depends on $B^\flat$.
This implies that
\begin{align*}
|B^\sharp\cap B^\flat|
&\le
\min(\rho^\sharp,\rho^\flat)
|\scriptc^\sharp\cap\scriptc^\flat|.
\\
&\le \frac{ \min(\rho^\sharp,\rho^\flat)} {\max(\rho^\sharp,\rho^\flat)}
\max(|B^\sharp|,|B^\flat|)
\end{align*}
If $ \frac {\max(\rho^\sharp,\rho^\flat)} {\min(\rho^\sharp,\rho^\flat)} 
\gtrsim\varrho(B^\sharp,B^\flat)$, this concludes the proof.
In the same way,
\begin{align*}
|B^\sharp\cap B^\flat|
&\le
\frac{|(\bar x^\sharp+\scriptc^\sharp) \cap (\bar x^\flat+\scriptc^\flat)|)}
{\max(|\scriptc^\sharp|,|\scriptc^\flat|)}
\max(|B^\sharp|,|B^\flat|).
\\
&\le
\frac{\min(|\scriptc^\sharp|,|\scriptc^\flat|)}{\max(|\scriptc^\sharp|,|\scriptc^\flat|)}
\max(|B^\sharp|,|B^\flat|).
\end{align*}
The quantity
$
\frac{\max(|\scriptc^\sharp|,|\scriptc^\flat|)}{\min(|\scriptc^\sharp|,|\scriptc^\flat|)}
$
is comparable to the sum of the second and third terms in the definition
of $\varrho(B^\sharp,B^\flat)$.
The desired inequality follows, if either of these terms
is $\gtrsim\varrho(B^\sharp,B^\flat)$.
Moreover
\[
|(\bar x^\sharp + \scriptc^\sharp)\cap (\bar x^\flat+\scriptc^\flat)|
\le C
\Big(
\sum_{j=1}^{d-1} |\langle \bar x^{\sharp}{}'-\bar x^{\flat}{}',e^\sharp_j\rangle|/r_j^\sharp
\ +\ 
\sum_{j=1}^{d-1} |\langle \bar x^{\sharp}{}'-\bar x^{\flat}{}',e^\flat_j\rangle|/r_j^\flat  
\Big)
\max(|\scriptc^\sharp|,|\scriptc^\flat|),
\]
so the conclusion holds if either the fourth or fifth term in the
definition of $\varrho$ is largest.

Consider next the final two terms in the definition.
Define 
$Q^\sharp(y) = y_d-(\bar x^\sharp_\star)_d - |y'-\bar x_\star^\sharp{}'|^2$ 
and
$Q^\flat(y) = y_d-(\bar x^\flat_\star)_d - |y'-\bar x_\star^\flat{}'|^2$.
Consider the case where the eighth term is large, in the sense that 
$|Q^\sharp(\bar x^\flat)|
\ge \tfrac1{2}\varrho(B^\sharp,B^\flat)\rho^\sharp$, and 
the first seven terms in the definition of $\varrho$
are all $\le c_0\varrho(B^\sharp,B^\flat)$ for a suitably small constant
$c_0$, and moreover, 
\[
\max(\rho^\sharp,\rho^\flat)/\min(\rho^\sharp,\rho^\flat)
\le \varrho(B^\sharp,B^\flat)^{1/2}.
\]
Then $\bar x^\flat\notin B^\sharp$. We aim to prove that
$|B^\sharp\cap B^\flat|$ is relatively small.
To this end consider the quadratic polynomial $P:\reals^{d-1}\to\reals$ 
defined by
$P(z) = Q^\sharp(z,t(z))$
where $t(z)$ is chosen so that
$Q^\flat(z,t(z))\equiv 0$, that is,
$t(z) = (\bar x^\flat_\star)_d + |z-\bar x_\star^\flat{}'|^2$.
Observe that 
\[\text{ If 
$|P(z)|> \rho^\flat+\rho^\sharp$
then $B^\sharp\cap B^\flat\cap (\{z\}\times\reals)=\emptyset$.}
\]

In the case which we are now analyzing, $P$ satisfies 
\[|P(\bar x^\flat)| 
\ge \tfrac1{2}\varrho(B^\sharp,B^\flat)\rho^\sharp.\]
Define 
\[\eps = 
\frac{3\max(\rho^\sharp,\rho^\flat)}
{\varrho(B^\sharp,B^\flat)\rho^\sharp}
\le 
3 \varrho(B^\sharp,B^\flat)^{-1/2}.
\]
On $\scriptc^\flat$,
\[
|P(z)|\ge \eps
\varrho(B^\sharp,B^\flat)\rho^\sharp =3\max(\rho^\sharp,\rho^\flat)> \rho^\sharp+\rho^\flat,
\]
for all $y$ outside a set of measure $\le C\eps^c|\scriptc^\flat|$.

We conclude that
\begin{align*}
|B^\sharp\cap B^\flat|
&\le C\eps^c\min(\rho^\sharp,\rho^\flat)|\scriptc^\flat|
\\
&\le C'\varrho(B^\sharp,B^\flat)^{-c'} |B^\flat|
\end{align*}
for some $c',C'\in\reals^+$. This is the bound required.

Everything is symmetric in the indices 
$\sharp,\flat$, so
the case where the ninth term is large requires no further discussion.
Likewise it suffices now to treat the case where the sixth term is
large, the seventh term being handled by symmetry. 

Suppose then that
$ |\langle \bar x_\star^{\sharp}{}'
-\bar x_\star^{\flat}{}',e^\sharp_i\rangle|
\ge c\varrho(B^\sharp,B^\flat)\rho^\sharp/r_i$ for some index $i$.
To simplify notation write $u = \bar x_\star^\sharp{}'$
and
$v = \bar x_\star^\flat{}'$.
If $y=(w,t)\in\reals^{d-1}\times\reals$ belongs to
$B^\sharp\cap B^\flat$ then
$t$ satisfies
$|t-(\bar x_\star^\sharp)_d-|w- u|^2|<\rho^\sharp$
and
$|t-(\bar x_\star^\flat)_d-|w- v|^2|<\rho^\flat$.
Subtracting gives
\[
\big| 2\langle w,u-v\rangle -s \big| < 2\max(\rho^\sharp,\rho^\flat)
\]
where $s=2|u|^2-2u\cdot v$. 
The value of $s$ is little consequence. 
With respect to the basis $\earrow^\sharp$,
the $i$-th component of $u-v$ has absolute value
$\ge c\varrho(B^\sharp,B^\flat)\rho^\sharp/r_i$.
$w$ ranges over a translate of $\scriptc^\sharp$;
in particular, $\langle w,\earrow^\sharp_i\rangle$
satisfies the sole constraint
$\langle w,\earrow^\sharp_i\rangle<r_i$. 
This forces
\[
\big|
\{w\in \bar x^\sharp{}'+\scriptc^\sharp:
\big| 2\langle w,u-v\rangle -s \big| < 2\max(\rho^\sharp,\rho^\flat)\}
\big|
\le C\varrho(B^\sharp,B^\flat)^{-1}|\scriptc^\sharp|,
\]
uniformly in all $s\in\reals$.
This implies the required bound, by a repetition of arguments
given above.
\end{proof}

There is an adequate quasi-triangle inequality for $\varrho$.
\begin{lemma} \label{lemma:triangle}
For any three paraballs,
\begin{equation}
\varrho(B^\sharp,B^\flat)\le C\varrho(B^\sharp,B^\natural)^C
+ C\varrho(B^\natural,B^\flat)^C
\end{equation}
where $C<\infty$ depends only on the dimension $d$.
\end{lemma}

By exploiting the transitive action of $\scriptg_d$, we may assume
without loss of generality that $B^\natural = B(z,\earrow,\rarrow,\rho)$
where $z=(0,0)$, $\earrow$ is the standard basis for $\reals^{d-1}$,
$\rarrow = (1,1,\cdots,1)$ and $\rho=1$. 
Reasoning like that in the proof of Lemma~\ref{lemma:ifparaballsintersect}
then contols the parameters specifying $B^\sharp$ in terms
of $\eta^{-1}=\varrho(B^\sharp,B^\natural)$,
and likewise for $B^\flat$. 
It then follows directly from the definitions that $\varrho(B^\sharp,B^\flat)
\le C\eta^{-C}$ where now $\eta$ is the maximum of the two quasidistances.

\subsection{Lorentz spaces}

\begin{definition}
Let $f$ be a nonnegative function which is finite almost everywhere.
A rough level set decomposition of $f$
is a representation
$f = \sum_{j=-\infty}^\infty
2^j f_j$ where $f_j$ is supported on a set $E_j$, $1\le |f_j(x)|<2$
for almost every $x$, and the sets $E_j$ are pairwise disjoint and measurable.
\end{definition}
\noindent
Any (complex-valued)
function $f$ which is finite almost everywhere on $\reals^d$
admits such a decomposition, which is unique modulo redefinition on sets
of Lebesgue measure zero. 
As shorthand for such a decomposition we will write
``$f = \sum_j 2^j f_j$, $f_j\leftrightarrow E_j$''.

A quasinorm for the Lorentz space $L^{p,r}(\reals^d)$
is 
\[
\norm{f}_{p,r} = \big(\sum_j (2^j|E_j|^{1/p})^r\big)^{1/r},
\]
with the natural interpretation for $r=\infty$.
For $p>1$ and $r\in [1,\infty]$ there exist equivalent expressions
which satisfy the triangle inequality. 
These spaces are nested: $L^{p,r}\subset L^{p,s}$ if $r\le s$,
and the inclusion is proper if $r<s$.
See \cite{steinweiss} for further information.

\subsection{Analytic preliminaries}

We review here four facts established in \cite{quasiextremal}.
The first of these results asserts that
any quasiextremal function for inequality \eqref{eq:paraboloid}
is closely connected with the characteristic function of
some paraball of comparable $L^{p }$ norm.

\begin{lemma} \cite{quasiextremal} \label{lemma:mainQEpaper}
For any $\eps>0$ there exist $c,C\in\reals^+$ with the following property.
If $f\in L^{p}$ is a nonnegative function with rough level set
decomposition $f = \sum_{j\in \integers} 2^j f_j$, $f_j\leftrightarrow E_j$,
and if $\norm{\tpar f}_q \ge \eps\norm{f}_p$
then there exist an index $j$ and a paraball
$B$ such that
\[
\norm{2^j f_j\cdot \chi_{B}}_p\ge c\eps^C\norm{f}_p
\]
and
\[
|B|\le |E_j|. 
\]
\end{lemma}
\noindent This is Theorem~1.5 of \cite{quasiextremal}.

It is often useful to work with the bilinear form $\langle g,\tpar f\rangle$.
The following result connects quasiextremal pairs $(f,g)$ with 
dual pairs of paraballs,
in the basic case when $f,g$ are both characteristic functions of sets.
\begin{proposition} \cite{quasiextremal} \label{prop:dualparaballs}
There exist positive finite constants $C,c,\gamma$, depending only on the dimension $d$,
with the following property.
Let $E,E^\star$ be measurable subsets of $\reals^d$ 
satisfying $0<|E|,|E^\star|<\infty$.
Define $\eps>0$ to be
\[
\eps=\frac{\langle \chi_{E^\star},\,\tpar \chi_E\rangle}{|E|^{d/(d+1)}|E^\star|^{d/(d+1)}}.
\]
Then there exists a pair $\scriptb=(B,B^\star)$ of dual paraballs
such that 
\begin{alignat*}{2}
&|B|\le |E|, 
&&|E\cap B|\ge c\eps^\gamma |E|,
\\
&|B^\star|\le |E^\star|,
\qquad
&&|E^\star\cap B^\star|\ge c\eps^\gamma |E^\star|.
\end{alignat*}
\end{proposition}

Under these hypotheses,
Theorem~1.2 of \cite{quasiextremal}
gives $|B|\le|E|$, $|B^\star|\le|E^\star|$,
and 
\[\langle \tpar(\chi_{E\cap B}),\chi_{E^\star\cap B^\star}\rangle
\ge c\eps^C\langle \tpar(\chi_E),\chi_{E^\star}\rangle
= c\eps^{C+1} |E|^{d/(d+1)}|E^\star|^{d/(d+1)}.\]
Since
$\langle \tpar(\chi_{E\cap B}),\chi_{E^\star\cap B^\star}\rangle
\le\apar|E\cap B|^{(d/(d+1)}|E^\star\cap B^\star|^{d/(d+1)}$,
the lower bounds 
$|E\cap B|\ge c\eps^C|E|$
and
$|E^\star\cap B^\star|\ge c\eps^C|E^\star|$
stated in Proposition~\ref{prop:dualparaballs} follow directly.

Although the inequality $\norm{\tpar f}_{q }\lesssim \norm{f}_{p }$
is dilation-invariant and is the only $L^p(\reals^d)\to L^q(\reals^d)$ inequality
valid for $\tpar$, it is nonetheless a suboptimal inequality within the
more general context of Lorentz spaces.  The following is Theorem~1.6 of \cite{quasiextremal}.
\begin{proposition}  \label{prop:lorentz}
$\tpar$ maps $L^{p ,r}$ boundedly to $L^{q }$ for all $p <r<q $.
\end{proposition}

The following was a principal ingredient in the proof of Proposition~\ref{prop:lorentz},
and will be needed again below. See Lemma 9.2 and inequality (9.10) of \cite{quasiextremal}.
\begin{lemma} \label{lemma:comparablemeasures}
For any $d\ge 2$ there exist $C,C'<\infty$ with the following
property.
Let $E,E',F\subset\reals^d$ be measurable sets with 
positive, finite measures.
Let $\eta\in (0,\apar]$.
If 
$T\chi_E(x)\ge \eta |E|^{1/p}|F|^{-1+1/p}$
and
$T\chi_{E'}(x)\ge \eta |E'|^{1/p}|F|^{-1+1/p}$
for every $x\in F$,
then $|E'|\le C'\eta^{-C}|E|$.
\end{lemma}
The exponents in the hypotheses are natural; 
the hypotheses imply for instance that
$\langle \chi_F,T\chi_E\rangle\ge \eta|F|^{1/p}|E|^{1/p}$.

\section{Distant paraballs interact weakly}

The following lemma may at present seem unmotivated, but will later provide,
in the proof of Lemma~\ref{lemma:secondspatial},
the geometric input for perhaps the most central step of our analysis.
By a partition of a set we will always mean 
an expression as a union of pairwise disjoint subsets.

\begin{lemma} \label{lemma:ballsfarapart}
For each $d\ge 2$ there exists $C<\infty$ with the following property.
Let $\eta\in(0,1]$.
Let $\{B_\alpha: \alpha\in S\}$ 
be an arbitrary finite collection of paraballs in $\reals^d$
satisfying $\varrho(B_\alpha,B_\beta)\ge C\eta^{-C}$ whenever $\alpha\ne \beta$.
Let $F\subset \reals^d$ be a Lebesgue measurable set of finite measure.
Then $F$ can be measurably partitioned as $F=\cup_{\alpha\in S} F_\alpha$ in such a way that
\begin{equation} \label{eq:bfa0}
\langle \chi_{F_\beta},\tpar \chi_{B_\alpha}\rangle\le\eta |F|^{1/p }|B_\alpha|^{1/p }
\text{ whenever $\alpha\ne \beta$.}
\end{equation}
\end{lemma}
\noindent
Here $C$ depends only on the dimension $d$, not on $\eta$.

\begin{proof}
Define \[\gamma_\beta = \tfrac13 \eta|F|^{-1+1/p}|B_\beta|^{1/p}\]
and
\[\tilde F_\beta = \{x\in F: \tpar \chi_{B_\beta}(x)>\gamma_\beta\},\]
noting that
\begin{equation} \label{eq:bfa1}
\langle \chi_{F\setminus \tilde F_\alpha},\,\tpar \chi_{B_\alpha}\rangle
\le \int_{F\setminus \tilde F_\alpha} \gamma_\alpha
\le \gamma_\alpha |F| 
= \tfrac13 \eta |F|^{1/p}|B_\alpha|^{1/p}.
\end{equation}
Choose pairwise disjoint measurable sets $F_\beta\subset\tilde F_\beta$
so that $\cup_\beta \tilde F_\beta = \cup_\beta F_\beta$.
Their union is not necessarily all of $F$, 
but $F^\dagger=F\setminus\cup_\beta F_\beta$ already satisfies 
$\langle \chi_{F^\dagger},\,\tpar\chi_{B_\alpha}\rangle\le\tfrac13 \eta|F|^{1/p}|B_\alpha|^{1/p}$
for every $\alpha$ by \eqref{eq:bfa1}, so 
it suffices to prove that for all $\alpha\ne\beta$,
\begin{equation} \label{eq:modifiedgoal}
\langle \chi_{F_\beta},\tpar\chi_{B_\alpha}\rangle
\le \tfrac23 \eta |F|^{1/p}|B_\alpha|^{1/p}.
\end{equation}

We prove \eqref{eq:modifiedgoal} by contradiction. Suppose that there exist indices
$\alpha\ne\beta\in S$ for which \eqref{eq:modifiedgoal} fails to hold.
These indices will remain fixed for the remainder of this proof.
We aim to prove that $\varrho(B_\alpha,B_\beta)$ is small,
contradicting the hypothesis.

Set $\scriptf=F_\beta\cap \tilde F_\alpha$.
Then $\langle \chi_{F_\beta\setminus\tilde F_\alpha},\,\tpar\chi_{B_\alpha}\rangle
\le \tfrac13 \eta|F_\beta|^{1/p}|B_\alpha|^{1/p}$ as in \eqref{eq:bfa1},
so 
\begin{equation}\label{eq:bfanegation}
\langle \chi_\scriptf,\,\tpar\chi_{B_\alpha}\rangle 
= 
\langle \big(\chi_{F_\beta}-\chi_{F_\beta\setminus\tilde F_\alpha}\big),
\,\tpar\chi_{B_\alpha}\rangle 
\ge \tfrac13 \eta |F|^{1/p}|B_\alpha|^{1/p}.
\end{equation}
Since $
\langle \chi_\scriptf,\,\tpar\chi_{B_\alpha}\rangle 
\le \apar|\scriptf|^{1/p}|B_\alpha|^{1/p}$ by definition
of $\apar$,
this forces
\[
|\scriptf|\ge 3^{-p}\eta^p\apar^{-p}|F|.
\]

Given $B_\alpha,B_\beta,\scriptf$ as above, apply Proposition~\ref{prop:dualparaballs}
with $E=B_\alpha$, $E^\star=\scriptf$ to obtain a pair
$\scriptb^\alpha
(z_\alpha,\earrow_\alpha,\rarrow_\alpha,\rho_\alpha)
=(B^\alpha,B_\star^\alpha)$
satisfying 
\begin{alignat*}{2}
&|B^\alpha|\le|B_\alpha|,
&&|B^\alpha_\star|\le |\scriptf|\le|F|,
\\
&|B^\alpha\cap B_\alpha|\ge c\eta^\gamma|B_\alpha|,
\qquad
&&|B^\alpha_\star\cap\scriptf|\ge c\eta^\gamma|\scriptf|\ge c\eta^\gamma|F|
\end{alignat*}
for certain constants $c,\gamma$,
whose values have changed from one occurrence to the next.


Set $\tilde\scriptf= \scriptf\cap B_\star^\alpha$.
We know already that $|\tilde\scriptf|\ge c\eta^\gamma|F|$.
Since $\tpar\chi_{B_\beta}(x)>\gamma_\beta$ for every
$x\in F_\beta\supset\scriptf\supset\tilde\scriptf$,
\begin{align*}
\langle \chi_{\tilde\scriptf},\tpar\chi_{B_\beta}\rangle
\ge \gamma_\beta|\tilde\scriptf|
= \tfrac13 \eta |\tilde\scriptf|\cdot|F|^{-1+1/p}|B_\beta|^{1/p}
\ge c\eta^\gamma |\tilde\scriptf|^{1/p}|B_\beta|^{1/p}
\end{align*}
for certain positive constants $c,\gamma$.
Consequently Proposition~\ref{prop:dualparaballs} can be applied again, 
this time with $E=B_\beta$ and $E^\star =\tilde\scriptf$,
to obtain a pair
$\scriptb^\beta
(z_\beta,\earrow_\beta,\rarrow_\beta,\rho_\beta)
=(B^\beta,B_\star^\beta)$
satisfying
\begin{alignat*}{2}
&|B^\beta|\le|B_\beta|,
\qquad&& |B^\beta_\star|\le |\tilde\scriptf|\le |F|,
\\
&|B^\beta\cap B_\beta|\ge c\eta^\gamma|B_\beta|,
\qquad&& |B^\beta_\star\cap\tilde\scriptf|\ge c\eta^\gamma|\tilde\scriptf|\ge c\eta^\gamma|F|
\end{alignat*}
for certain constants $c,\gamma$.

Since 
$B_\star^\beta\cap B_\star^\alpha\supset B_\star^\beta\cap B_\star^\alpha\cap\scriptf
= B_\star^\beta\cap\tilde\scriptf$,
\begin{equation*}
|B_\star^\beta\cap B_\star^\alpha|
\ge |B_\star^\beta\cap\tilde\scriptf|
\ge c\eta^\gamma|F|
\ge c\eta^\gamma \max(B_\star^\beta,B_\star^\alpha).
\end{equation*}
By Lemmas~\ref{lemma:ifparaballsintersect} and \ref{lemma:dualballsclose},
this implies that
\[
\varrho(B^\alpha,B^\beta)\le C\eta^{-C}.
\]
Since $|B^\alpha|\le|B_\alpha|$ and 
$|B^\beta|\le|B_\beta|$,
while 
$|B^\alpha\cap B_\alpha|\ge c\eta^\gamma |B_\alpha|$ 
and
$|B^\beta\cap B_\beta|\ge c\eta^\gamma |B_\beta|$,
one has
\[
\varrho(B^\beta,B_\beta)\le C\eta^{-C}
\text{ and }
\varrho(B^\alpha,B_\alpha)\le C\eta^{-C}\]
by Lemma~\ref{lemma:ifparaballsintersect}.
Therefore by the quasi-triangle inequality of
Lemma~\ref{lemma:triangle}, 
\[
\varrho(B_\alpha,B_\beta)\le C\eta^{-C}.
\]
This contradicts the assumption that $\varrho(B_\alpha,B_\beta)$ is sufficiently large.
\end{proof}

\section{Step 1: Entropy refinement}

According to Proposition~\ref{prop:lorentz}, the inequality
\eqref{eq:paraboloid}, while scale-invariant, is not sharp within
the scale of Lorentz spaces. From this lack of optimality there follows useful information. 

\begin{lemma} \label{lemma:lorentz}
For any $d\ge 2$ 
there exists $C<\infty$ with the following property.
Let $\eps>0$.
Let $f$ be any nonnegative measurable function in $L^{p }(\reals^d)$.
Then there exist an index set $S\subset\integers$
of cardinality $|S|\le C\eps^{-C}$
and a function $\tilde f$ satisfying $0\le \tilde f\le f$
with rough level set decomposition
$\tilde f = \sum_{j\in S}2^jf_j$, 
such that
\begin{equation}
\norm{\tpar\tilde f}_{q }\ge (1-\eps)\norm{\tpar f}_{q }.
\end{equation}
\end{lemma}

\begin{proof}
Choose any $r\in (p ,q )$.
Let $f$ have a rough level set decomposition $f\equiv \sum_{j\in\integers}2^j f_j$,
$f_j\leftrightarrow E_j$. 
Let $\eta>0$ be a small parameter.
Define 
\[S=\{j: 2^{j}|E_j|^{1/p }>\eta\}
\text{ and } \tilde f =\sum_{j\in S} 2^jf_j.\]
Then
\begin{multline*}
\norm{f-\tilde f}_{p,r}^r
= \sum_{j\notin S} (2^{j}|E_j|^{1/p})^r
= \sum_{j\notin S} 
(2^j|E_j|^{1/p})^p
(2^j|E_j|^{1/p})^{r-p}
\\
\le \eta^{r-p} \sum_{j\in\integers}
(2^j|E_j|^{1/p})^p
= \eta^{r-p}\norm{f}_{p}^p
\end{multline*}
by H\"older's inequality.
Therefore 
\[
\norm{\tpar(f-\tilde f)}_{q }
\le C \norm{f-\tilde f}_{p,r}
\le C\eta^{1-p/r}\norm{f}_p^{p/r}
\]
where $C<\infty$ is the norm
of $\tpar$ as an operator from $L^{p ,r}$ to $L^{q }$.
Moreover,
\[
\eta^{p}|S|= \sum_{j\in S} \eta^p \le \sum_{j} 2^{jp}|E_j|\le \norm{f}_p^p.
\]
We may assume without loss of generality that $\norm{f}_p=1$.
Then defining $\eta$ to satisfy $C\eta^{1-p/r}=\eps$ gives the conclusion stated.
\end{proof}

\begin{lemma} \label{lemma:lorentzsecondversion}
Let $f\ge 0$ satisfy $\norm{\tpar f}_q\ge (1-\delta)\apar\norm{f}_p$.
Then the function $\tilde f$ in Lemma~\ref{lemma:lorentz}
can be chosen to satisfy 
\begin{equation}
\norm{f-\tilde f}_p
\le C\, (\eps+\delta)^{1/p}\,\norm{f}_p
\end{equation}
in addition to all the conclusions of Lemma~\ref{lemma:lorentz}.
\end{lemma}

\begin{proof}
Construct $\tilde f$ as in the proof of Lemma~\ref{lemma:lorentz}.
Since \[\norm{\tpar\tilde f}_q\le \apar\norm{\tilde f}_p\]
and \[\norm{\tpar\tilde f}_q\ge (1-\eps)\norm{\tpar f}_q,\]
it follows that
\[\norm{\tilde f}_p \ge \apar^{-1}(1-\eps)\norm{\tpar f}_q.\]
Since $\tilde f,f-\tilde f$ have disjoint supports,
\begin{align*}
\norm{f-\tilde f}_p^p 
&= \norm{f}_p^p-\norm{\tilde f}_p^p
\\
&\le \norm{f}_p^p-(1-\eps)^p\apar^{-p}\norm{\tpar f}_q^p
\\
&\le \norm{f}_p^p-(1-\eps)^p\apar^{-p}(1-\delta)^p\apar^p\norm{f}_p^p
\\
&= \big[1-(1-\eps)^p(1-\delta)^p\big] \norm{f}_p^p
\\
&\le (C\eps+C\delta)\norm{f}_p^p.
\end{align*}
\end{proof}

The upshot is that near-extremals have low entropy, in the sense
that relatively few terms in their rough level set decompositions
suffice to approximate them to a specified degree of accuracy.

We have implicitly also established the following 
variant of Lemma~\ref{lemma:lorentzsecondversion}.
\begin{lemma} \label{lemma:lorentzthirdversion}
There exists $c,C\in\reals^+$ with the following property.
Suppose that $0\le f\in L^p$, and 
let $f$ 
satisfy $\norm{\tpar f}_q\ge (1-\delta)\apar\norm{f}_p$,
and have rough level set decomposition $f=\sum_j 2^j f_j$, $f_j\leftrightarrow E_j$.
Then for any $\eta\in(0,1]$,
\begin{equation*}
\Big\|
\sum_{j: 2^j|E_j|^{1/p}<\eta\norm{f}_p} 2^j f_j
\Big\|_p\le C(\delta^{1/p}+\eta^c)\norm{f}_p.
\end{equation*}
\end{lemma}

\section{Step 2: Weak higher integrability}

In Lemma~\ref{lemma:lorentz} we obtained a very weak form of
precompactness, in the form of an {\it a priori} bound on the
cardinality of the index set $S$ in the sum
$\tilde f= \sum_{j\in S} 2^jf_j$.
The next step is to show that the indices in $S$ cannot be 
far apart from one another. 

\begin{lemma} \label{lemma:levelsfarapart} \label{lemma:step2}
There exist constants $c,C,\tilde C\in(0,\infty)$, depending only
on the dimension $d$, with the following property.
Let $\rho\in(0,1)$.
Let $f$ be a $(1-\delta)$-quasiextremal for \eqref{eq:paraboloid}
satisfying $\norm{f}_{p }=1$.
If $\delta\le c\rho^C$,
then there exists a function $\tilde f$ 
satisfying $\norm{f-\tilde f}_{p }\le C\rho^c$
with a rough level set decomposition
$\tilde f=\sum_{j} 2^j f_j$
such that if
both $\norm{2^i f_i}_p\ge\rho$
and $\norm{2^j f_j}_p\ge\rho$,
then
\[|i-j|\le \tilde C
\rho^{-\tilde C}\]
\end{lemma}


\begin{proof}
By Lemma~\ref{lemma:lorentz}
there exists a function $\tilde f$ with rough level set decomposition
$\tilde f = \sum_{j\in S} 2^j f_j$, $f_j\leftrightarrow E_j$,
with $|S|\le C\rho^{-C}$
such that $\norm{f-\tilde f}_{p}\le C\rho^c$,
and moreover
$\norm{2^j f_j}_{p}\ge\rho$ for every $j\in S$.

The operator dual to $\tpar$ is identical to $\tpar$ under conjugation
by a simple change of variables. Thus 
Lemma~\ref{lemma:lorentz} applies equally well to it.
Therefore by duality,
there exists $h\in L^{p}$ satisfying $\norm{h}_{p}=1$,
with rough level set decomposition $h = \sum_{k\in \tilde S} 2^k h_k$,
$h_k\leftrightarrow F_k$,
with index set 
$\tilde S\subset\integers$ satisfying $|\tilde S|\le C\rho^{-C}$,
so that 
\[\langle h,\tpar\tilde f\rangle\ge (1-2\delta)\apar.\]
Set $N=|S|+|\tilde S|\le C\rho^{-C}$.

Set \[M = \max_{i,j\in S} |i-j|.\]
It is possible to
partition $S$ into two nonempty disjoint sets $S=S^\sharp\cup S^\flat$
so that 
$|i-j|\ge M/N$
whenever $i\in S^\sharp$ and $j\in S^\flat$.
Fix any such partition.

Let $\eta>0$ be a small parameter to be chosen below; it will depend
on $N$ and thereby ultimately on $\rho$.
Partition each of the sets $F_k$ as
$F_k=F_k^\sharp\cup F_k^\flat\cup F_k^\natural$
measurably, so that
\newline
(i) For each $x\in F_k^\sharp$ there exists $j\in S^\sharp$
such that $\tpar\chi_{E_j}(x)>\eta|F_k|^{-1+1/p}|E_j|^{1/p}$;
\newline
(ii) For each $x\in F_k^\flat$ there exists $j\in S^\flat$
such that $\tpar\chi_{E_j}(x)>
\eta|F_k|^{-1+1/p}|E_j|^{1/p}$;
\newline
(iii) For each $x\in F_k^\natural$, $\tpar\chi_{E_j}(x)
\le\eta|F_k|^{-1+1/p}|E_j|^{1/p}$ for every $j\in S$.

Set 
\begin{alignat*}{2}
&h^\flat = \sum_{k\in\tilde S} 2^k h_k\chi_{F_k^\flat},
\qquad&& h^\sharp= \sum_{k\in\tilde S} 2^k h_k\chi_{F_k^\sharp},
\\
&f^\flat = \sum_{j\in S^\flat} 2^j f_j,
\qquad&& f^\sharp= \sum_{j\in S^\sharp} 2^j f_j.
\end{alignat*}
For any $j\in S$ and $k\in\tilde S$,
\begin{align*}
\langle 2^k h_k\cdot \chi_{F_k^\natural},\tpar (2^j f_j)\rangle
&\le 2^{k+j+2}\langle \chi_{F_k^\natural},\tpar \chi_{E_j}\rangle
\\
&\le 2^{k+j+2}\eta|F_k^\natural|\cdot|F_k|^{-1+1/p}|E_j|^{1/p}
\\
&\le 2^{k+j+2}\eta|F_k|^{1/p}|E_j|^{1/p}
\\
&\le 4\eta\norm{h}_p\norm{f}_p
\\
&=4\eta.
\end{align*}

$\tilde f=f^\sharp+f^\flat$ where $f^\sharp,f^\flat$ have disjoint supports;
likewise $h= h^\sharp+h^\flat+h^\natural$, with  disjointly supported
summands. 
Consequently, as in the proof of Lemma~\ref{lemma:lorentz}, 
\[
 \langle h^\sharp,\tpar f^\sharp\rangle
+ \langle h^\flat,\tpar f^\flat\rangle
\le \apar(1-C\rho^c)\norm{f}_p\norm{h}_p
= \apar(1-C\rho^c)
\]
since $\norm{f^\sharp}_p\ge\rho$, 
$\norm{f^\flat}_p\ge\rho$,
$1=\norm{h}_p^p \ge \norm{h^\sharp}_p^p +\norm{h^\flat}_p^p$,
and
$1 = \norm{f}_p^p \ge \norm{\tilde f}_p^p = \norm{f^\sharp}_p^p +\norm{f^\flat}_p^p$.
Thus 
\begin{align*}
\langle h,\tpar \tilde f\rangle
&\le CN^2\eta 
+ \langle h^\sharp,\tpar f^\sharp\rangle
+ \langle h^\flat,\tpar f^\flat\rangle
+ \langle h^\flat,\tpar f^\sharp\rangle
+ \langle h^\sharp,\tpar f^\flat\rangle
\\
&\le CN^2\eta 
+ \apar(1-C\rho^c)
+ \langle h^\flat,\tpar f^\sharp\rangle
+ \langle h^\sharp,\tpar f^\flat\rangle.
\end{align*}

Set 
\begin{equation}
\eta =c_0 \rho^{C_0},
\end{equation}
where $c_0,C_0$ are respectively sufficiently small and sufficiently large
constants. In particular, choose $c_0,C_0$ so that
$CN^2\eta + (\apar-C\rho^c) < (1-3\delta)\apar $;
this is possible
since $N\le C\rho^{-C}$. 
Thus 
\begin{equation} \label{eq:antidiagonal} 
\langle h^\flat,\tpar f^\sharp\rangle\ge c\rho^C,
\end{equation}
or
\begin{equation*} \langle h^\sharp,\tpar f^\flat\rangle\ge c\rho^C.
\end{equation*}
It is no loss of generality to assume \eqref{eq:antidiagonal}, 
since the situation is symmetric with respect to interchange of the
indices $\sharp,\flat$.

There must exist $k\in\tilde S$ and $j\in S^\sharp$
such that 
\begin{equation}
\langle \chi_{F_k^\flat},\tpar \chi_{E_j}\rangle\ge \eta |F_k|^{1/p}|E_j|^{1/p}; 
\end{equation}
otherwise we would have
a total bound of $C\eta N^2$ for the left-hand side in 
\eqref{eq:antidiagonal}, which would be a contradiction
due to the choice of $\eta$.

Let 
\[F_k^{\flat,\dagger}=\{x\in F_k^\flat: T\chi_{E_j}(x)<\tfrac12\eta 
|F_k^\flat|^{-1}|F_k|^{1/p}|E_j|^{1/p}\}.\]
Then 
\[\langle \chi_{F_k^{\flat,\dagger}},T\chi_{E_j}\rangle
\le\tfrac12 \eta |F_k|^{1/p}|E_j|^{1/p},\]
so
\[\langle \chi_{F_k^\flat\setminus F_k^{\flat,\dagger}},T\chi_{E_j}\rangle
\ge\tfrac12 \eta |F_k|^{1/p}|E_j|^{1/p}.\]
Since 
\[ \langle \chi_{F_k^\flat\setminus F_k^{\flat,\dagger}},T\chi_{E_j}\rangle
\le \apar|F_k^\flat\setminus F_k^{\flat,\dagger}|^{1/p}|E_j|^{1/p},\]
we must have
\[|F_k^\flat\setminus F_k^{\flat,\dagger}|\ge c\rho^C|F_k|.\]
Thus 
\[T\chi_{E_j}(x)\ge c\rho^C |F_k|^{-1+1/p}|E_j|^{1/p}\]
for every $x\in F_k^\flat\setminus F_k^{\flat,\dagger}$.

There exist an index $i\in S^\flat$
and a subset $\scriptf\subset F_k^\flat\setminus F_k^{\flat,\dagger}$
such that $\tpar\chi_{E_i}(x)>\eta|F_k|^{-1+1/p}|E_i|^{1/p}$
for all $x\in\scriptf$
and 
\[\langle \chi_{\scriptf},\tpar\chi_{E_j}\rangle
\ge N^{-1}
\langle \chi_{F_k^\flat\setminus F_k^{\flat,\dagger}},\tpar\chi_{E_j}\rangle
\ge cN^{-1}\rho^C|F_k|^{1/p}|E_j|^{1/p}.\]
Since $\langle \chi_{\scriptf},\tpar\chi_{E_j}\rangle
\le\apar|\scriptf|^{1/p}|E_j|^{1/p}$,
it follows again that $|\scriptf|\ge c\rho^C|F_k|$.

Therefore 
\begin{align*}
&T\chi_{E_j}(x)\ge c\rho^C|\scriptf|^{-1+1/p}|E_j|^{1/p}
\\
&T\chi_{E_i}(x)\ge c\rho^C|\scriptf|^{-1+1/p}|E_i|^{1/p}
\end{align*}
for every $x\in\scriptf$.
By Lemma~\ref{lemma:comparablemeasures},
this forces 
\[|E_i|\le C\rho^{-C}|E_j| \text{ and } |E_j|\le C\rho^{-C}|E_i|.\]

Since $\norm{f_i}_p\ge\rho$,
$\rho^p\le 2^p 2^{pi}|E_i|$.
On the other hand, 
$2^{p(j+1)}|E_j|\le \norm{\tilde f}_p^p=1$.
Therefore \[|E_j|\le C2^{(i-j)p}\rho^{-p}|E_i|.\]
In conjunction with the reverse bound $|E_j|\ge c\rho^C|E_i|$
proved above, this forces $|i-j|\le C\log(1/\rho)$.

This conclusion holds for a certain pair $(i,j)\in S^\flat\times S^\sharp$.
The partition $S = S^\sharp\cup S^\flat$ was chosen so that
$M\le N|i-j|$. 
Therefore $M\le C\rho^{-C}$.
\end{proof}

\begin{remark}
This proof is on track to establish higher integrability in the natural form
$\phi^*f\in L^Q$ for some $Q>p $ for extremals, 
until the very last step, in which $|i-j|$ turns into $N|i-j|$.
Perhaps some more efficient reorganization is possible.
\end{remark}

\begin{corollary} \label{cor:Psi}
There exist a finite constant $C$ and a 
function $\Psi:(0,\infty)\to(0,\infty)$
satisfying 
$\Psi(t)/t^p\to\infty$ as $t\to\infty$
and as $t\to 0$,
with the following property.
For any $\eps>0$ there exists $\delta>0$ such that
for any nonnegative function $f$ satisfying
$\norm{f}_p=1$
and $\norm{\tpar f}_q\ge(1-\delta)\apar$,
there exist $\phi\in\scriptg_d$ and
a decomposition $\phi^*f=g+h$
such that $g,h\ge 0$ satisfy
$\norm{h}_p<\eps$
and 
\begin{equation}
\int \Psi(g)\le C.
\end{equation}
\end{corollary}

\begin{proof}
Fix any exponent $r\in(p,q)$.
Let $\eta>0$  be a small quantity to be chosen at the end
of the proof. Let $\delta$ be a sufficiently small function of $\eta,\eps$.

Let $f$ have a rough level set decomposition $f\equiv \sum_{j\in\integers}2^j f_j$,
$f_j\leftrightarrow E_j$. 
By Lemma~\ref{lemma:lorentzthirdversion}, if $\delta>0$ is sufficiently small then
there exists at least one index $k$ such that
$\norm{2^kf_k}_p\ge c_0$,
where $c_0>0$ is an absolute constant which depends only on the dimension $d$. 
By choosing $\phi\in\scriptg_d$ to be an appropriate dilation symmetry
$\phi(x',x_d) = (rx',r^2x_d)$
we may reduce to the case where $k=0$.

By Lemma~\ref{lemma:levelsfarapart},
there exists $M<\infty$ such that
$\norm{2^j f_j}_p<\eta$ 
whenever $|j|\ge M$.
Define $h=\sum_{|j|>M} 2^j f_j$,
and
$g=\sum_{|j|\le M} 2^j f_j$.
Then
$\norm{h}_{L^{p,r}}^r
= \sum_{|j|>M}\norm{2^jf_j}_p^r
\le \eta^{r-p}
\sum_{|j|>M}\norm{2^jf_j}_p^p
\le \eta^{r-p}$.
By the proof of Lemma~\ref{lemma:lorentzsecondversion},
this implies that $\norm{h}_p<\eps$
provided that $\delta,\eta$ are chosen to be
sufficiently small.

If $\eta,\delta$ are chosen to be sufficiently small
then by Lemma~\ref{lemma:levelsfarapart},
for any $\rho\ge\eta$,
$|j|\le C\rho^{-C}$ whenever $\norm{2^jf_j}_p\ge\rho$.
Since the sets $E_j$ are pairwise disjoint,
for any nondecreasing function $\Psi$,
$\int \Psi(g) = \sum_{|j|\le M} \int \Psi(2^jf_j)
\le \sum_{|j|\le M} \Psi(2^{j+1})|E_j|$.

Let $S_k=\{j: |j|\le M \text{ and } \norm{2^jf_j}_p\in (2^{-k-1},2^{-k}]\}$.
Then for all $j\in S_k$, $|j|\le C2^{Ck}$ by Lemma~\ref{lemma:levelsfarapart}.
Therefore
\[
\sum_{j\in S_k} \Psi(2^{j+1})|E_j|
\le 
2^p \max_{|j|\le C2^{Ck}} \frac{\Psi(2^{j+1})}{2^{p(j+1)}} 
\cdot \sum_{j\in S_k} \norm{2^j f_j}_p^p.
\]
Moreover,
\[
\sum_{j\in S_k} \norm{2^j f_j}_p^p
\le C\delta+C2^{-ck}
\]
by Lemma~\ref{lemma:lorentzthirdversion}.
Thus
\[
\sum_{2^k\le\eta^{-1}}
\sum_{j\in S_k} \Psi(2^{j+1})|E_j|
\le 
\sum_{2^k\le\eta^{-1}}
C
(\delta+2^{-ck})
\max_{|j|\le C2^{Ck}} 
\Big(\frac{\Psi(2^{j+1})}{2^{p(j+1)}}\Big).
\]
Choose $\Psi$ 
to be a nondecreasing function satisfying the growth condition
\[
\sum_{k=0}^\infty 2^{-ck}
\max_{|j|\le C2^{Ck}} \frac{\Psi(2^{j+1})}{2^{p(j+1)}} 
<\infty,
\]
and
$\Psi(t)/t^p\to\infty$
as $t\to 0$ and as $t\to\infty$.
Then choose $\delta$ to be a function of $\eta$, 
satisfying
\[
\delta \sum_{2^k\le\eta^{-1}}
\max_{|j|\le C2^{Ck}} \frac{\Psi(2^{j+1})}{2^{p(j+1)}} 
\le 1.
\]
This completes the proof.
\end{proof}

\section{Step 3: Spatial localization}

Lemma~\ref{lemma:mainQEpaper}
leads to the following preliminary result concerning
the geometric structure of quasiextremals.

\begin{lemma} \label{lemma:firstspatial}
For any $\eps>0$ there exist $\delta>0$ and $N,K<\infty$
with the following property.
Let $f\ge 0$ be any $(1-\delta)$-quasiextremal for
inequality \eqref{eq:paraboloid} satisfying $\norm{f}_{p }=1$.
Then there exists a function $F$ 
with a rough level set decomposition
$F = \sum_{j\in S} 2^j F_j$, $F_j\leftrightarrow E_j$,
satisfying
\begin{gather}
0\le F\le f
\notag
\\
\norm{\tpar F}_{q }\ge (1-\eps)\apar,
\notag
\\
|i-j|\le K \text{ for all $i,j\in S$,}
\notag
\\
\intertext{
and for each $j\in S$ there exist 
$N$ paraballs $B_{j,i}$ such that}
E_j\subset \cup_{i=1}^N B_{j,i}
\label{eq:pb1}
\\
\sum_i|B_{j,i}|\le C(\eps)|E_j|.
\label{eq:pb2}
\end{gather}
\end{lemma}

\begin{proof}
In light of results already proved,
by modifying $f$ by a function whose $L^p$ norm is $<\eps/2$,
we may suppose that $f$
has a finite rough level set decomposition
$f=\sum_{j\in S} 2^j f_j$, $f_j\leftrightarrow\tilde E_j$, 
with $|i-j|\le K(\eps)$ for all $i,j\in S$.
Let $\eta>0$ be a small quantity. Then $\norm{\tpar f}_q
\ge (1-\delta)\apar\ge \eta$, so we may apply Lemma~\ref{lemma:mainQEpaper}
to find a paraball $B$ and an index $i_1$ such that
$\norm{2^{i_1} f_{i_1}\chi_B}_p\ge \rho$ 
and $|B|\le |\tilde E_{i_1}|$,
where $\rho>0$ depends only on $\eta$.
Set $g_1 = 2^{i_1}f_{i_1}\chi_B$ and write $f = g_1 + h_1$.

This was step $1$ of a construction which we iterate,
as follows. At step $n$, one is given 
a collection of paraballs $\{B_m: m\le n-1\}$, along with
a decomposition
\[
f = \sum_{m=1}^{n-1} g_m
+ h_{n-1} 
\]
such that
the sets $\tilde B_m = B_m\setminus\cup_{l<m} B_l$ 
and functions $g_m,h_{n-1}$ satisfy
\begin{gather*}
 g_m=2^{i_m}f_{i_m}\chi_{\tilde B_m} \text{ for some $i_m\in\integers$,} 
\\
|B_m|\le |\tilde E_{i_m}|,
\\
h_{n-1} = f\cdot\chi_{\reals^d\setminus\cup_{m\le n-1} B_m}.
\end{gather*}
Since $0\le h_{n-1}\le f$,
$\norm{h_{n-1}}_p\le \norm{f}_p\le 1$.
If $\norm{\tpar h_{n-1}}_q<\eta$ then the construction terminates.
Otherwise invoke Lemma~\ref{lemma:mainQEpaper} to find a paraball $B_n$ 
and an index $i_n$ 
such that the function
\[
g_n=2^{i_n}f_{i_n}\cdot\chi_{\reals^d\setminus\cup_{m<n} B_m}\cdot \chi_{B_n}
\]
satisfies
$\norm{g_n}_p\ge\rho$
and
$|B_n|\le |\tilde E_{i_n}|$.

The functions $g_n$ are nonnegative and have pairwise disjoint supports, 
and $\sum_{m=1}^n g_m\le f$. Consequently
$\sum_{m=1}^n \norm{g_m}_p^p\le \norm{f}_p^p=1$,
so this process must terminate after at most $\rho^{-p}$ iterations.
If it terminates at the $n$-th step, then
set $F = \sum_{m=1}^{n-1}g_m$. 
Then $\norm{\tpar(f-F)}_q\le\eta$, so
by Lemma~\ref{lemma:lorentzsecondversion}, $\norm{f-F}_p<\tfrac12\eps$
provided that $\eta$ is chosen to be sufficiently small.
Defining the collection $\{B_{j,i}\}$ to be $\{\tilde B_m: i_m=j\}$ for each
index $j$ produces a collection of paraballs satisfying \eqref{eq:pb1},\eqref{eq:pb2}.
\end{proof}

Consider momentarily the possibility of a sequence $\{f_\nu\}$
of quasiextremals which are characteristic functions of sets
$E_\nu$ satisfying $|E_\nu|=1$. 
If these sets were to move off to spatial infinity as $\nu\to\infty$,
then $f_\nu$ and $f_\nu^p$ would converge weakly to zero, preventing the extraction
of an extremal as a limit of some subsequence.
If each $E_\nu$ were a paraball, 
then this situation could be rectified by invoking
the symmetry group ${\scriptg}_d$ to replace each $E_\nu$ by a 
a paraball independent of $\nu$.
However, there is potentially a more problematic obstruction:
If each $E_\nu$ were a disjoint union $E_\nu=E'_\nu\cup E''_\nu$,
with $|E'_\nu|=|E''_\nu|=\tfrac12$ 
and with $E'_\nu,E''_\nu$ moving to infinity in different directions, 
then the symmetries would not suffice to produce a useful renormalized sequence.
The following refinement of Lemma~\ref{lemma:firstspatial} 
rules out this sort of obstruction.

\begin{lemma} \label{lemma:secondspatial}
For any $\eps>0$ there exist $\delta>0$ and $K,\lambda<\infty$
with the following properties.
Let $f\ge 0$ be any $(1-\delta)$-quasiextremal for
inequality \eqref{eq:paraboloid} satisfying $\norm{f}_{p }=1$.
Then there exist a function $\tilde f$ and a paraball $B$
such that 
\begin{gather*}
0\le\tilde f\le f,
\\
\norm{\tilde f}_{p }\ge 1-\eps,
\\
\norm{T\tilde f}_{q }\ge (1-\eps)\apar,
\intertext{and $\tilde f$ admits a rough level set decomposition
$\tilde f = \sum_{j\in S} 2^j f_j$ with a distinguished index $J\in\integers$ satisfying}
|j-J|\le K \text{ for all } j\in S,
\\
f_j \text{ is supported in } \lambda B \text{ for all $j\in S$,}
\\
2^J|B|^{1/p}\le C\norm{f}_p.
\end{gather*}
\end{lemma}

This improves upon
Lemma~\ref{lemma:firstspatial}
in that the collection of paraballs
$B_{j,i}$ in the conclusion has been replaced by a single
paraball, which however must be expanded by the factor $\lambda$.
This is a geometric analogue of the replacement of
an upper bound on the cardinality of the index set $S$
in Lemma~\ref{lemma:lorentzsecondversion} 
by an upper bound on its diameter in Lemma~\ref{lemma:step2}.
Lemma~\ref{lemma:secondspatial} follows directly from the combination of
Lemmas~\ref{lemma:ballsfarapart} and \ref{lemma:firstspatial},
in the same way that Lemma~\ref{lemma:step2}
followed from Lemma~\ref{lemma:lorentzsecondversion} 
combined with Lemma~\ref{lemma:comparablemeasures}.
Details are left to the reader.


\section{Weak Convergence and Extremizers}

Let $f$ be 
a $(1-\delta)$-quasiextremal 
satisfying $\norm{f}_{p }=1$
for $\delta$ sufficiently small.
Apply Lemma~\ref{lemma:secondspatial}.
Then by replacing $f$
by $\phi^*f$ for an appropriately chosen $\phi\in\scriptg_d$,
we may reduce to the case where in the conclusions of 
that lemma,
$J=0$ and the paraball $B$ is 
$B=\{x\in\reals^d: |x_j|\le 1
\text{ for all } 1\le j\le d-1
\text{ and } |x_d-|x'|^2|<1\}$.
Thus we have the following  information.

\begin{lemma} \label{lemma:summary}
There exist a constant $C_0<\infty$ and 
functions $\Psi:[0,\infty)\to[0,\infty)$ and $\rho:[1,\infty)\to(0,\infty)$
satisfying $\Psi(t)\ge t^{p }$ for all $t$,
$\Psi(t)/t^{p }\to\infty$ as $t\to 0$ and as $t\to\infty$, 
and $\rho(R)\to 0$ as $R\to\infty$,
such that
for any $\eps>0$ there exists $\delta>0$ with the following property.
Let $f$ be any nonnegative function satisfying 
$\norm{f}_{p }=1$
and
$\norm{\tpar f}_{q }\ge (1-\delta)\apar$.
Then there exist $\phi\in\scriptg_d$ and a decomposition
\[\phi^* f = g+h\]
with $g,h\ge 0$ satisfying
\begin{gather*}
\norm{h}_{p }<\eps,\qquad
\int_{\reals^d} \Psi(g)\le C_0,\qquad
\int_{|x|\ge R}g^p\le \rho(R).
\end{gather*}
Moreover, there exists a nonnegative function
$F$ satisfying
$\norm{F}_{p }=1$ and
\begin{gather*}
\langle F,\tpar g\rangle\ge (1-\eps)\apar,
\qquad
\int_{\reals^d} \Psi(F)\le C_0,
\qquad
\int_{|x|\ge R}F^p\le \rho(R).
\end{gather*}
\end{lemma}
The bound for $\int\Psi(g)$ is Corollary~\ref{cor:Psi}. 
The bound for $\int_{|x|\ge R}g^p$ follows in a similar
way from Lemma~\ref{lemma:secondspatial}.

With these uniform bounds in hand, it is straightforward to 
derive most of the conclusions of Theorem~\ref{thm:main}.
For any sequence of functions $\{f_\nu\}$,
we write $f_\nu\rightharpoonup f$
to mean that $\int f_\nu\varphi\to \int f\varphi$
as $\nu\to\infty$,
for every compactly supported continuous test functions
$\varphi$.
If $f_\nu$ are nonnegative $L^1$ functions,
then $f_\nu\rightharpoonup f$ implies $\norm{f_\nu}_1\to \norm{f}_1$,
provided that
$\sup_\nu\int_{|x|\ge R}f_\nu\to 0$ as $R\to\infty$,

\begin{proof}[Proof of existence of extremizers.]
Let $\{f_\nu\}$ be any extremizing sequence. Then by the preceding lemma,
there exist $\phi_\nu\in\scriptg_d$ such that $\phi_\nu^* f_\nu
= g_\nu+h_\nu$ where 
$\norm{h_\nu}_p\to 0$,
while the functions $g_\nu$ satisfy 
the other conclusions of Lemma~\ref{lemma:summary};
and there exist corresponding functions $F_\nu$ satisfying 
the same bounds as $g_\nu$, with
$\norm{F_\nu}p=1$ and
$\langle F_\nu,\tpar g_\nu\rangle \to \apar$.

It follows directly from the Banach-Alaoglu theorem
that after passage to a subsequence of the index $\nu$,
$F_\nu^p\rightharpoonup F^p$
and
$g_\nu^p\rightharpoonup g^p$
form some $F,g\in L^p$.
Since $\int F_\nu^p\eta\le 1$ for every continuous, compactly supported
function $\eta$ satisfying $\norm{\eta}_\infty\le 1$,
necessarily $\int F^p\le 1$; likewise $\int g^p\le 1$.

We claim that
\begin{equation} \label{weaklimit}
\langle F_\nu,\tpar g_\nu\rangle \to \langle F,\tpar g\rangle. 
\end{equation}
Since $\norm{F}_p\le 1$ and likewise $\norm{g}_p\le 1$
and $\langle F_\nu,\tpar g_\nu\rangle\to \apar$,
it follows that $\norm{F}_p=\norm{g}_p=1$
and that $g$ is an extremal.

To prove \eqref{weaklimit} define
\[g_{\nu,\lambda}(x) = g_\nu(x)\chi_{|x|\le\lambda}(x)\chi_{g_\nu(x)\le\lambda}(x),\]
and define $g^{(\lambda)},F_{\nu,\lambda},F^{(\lambda)}$ in the corresponding way.
For any compactly supported function $\eta\in C^1_0(\reals^d)$,
the operator $f\mapsto \eta\tpar \eta f$
is smoothing in the sense that it
maps $L^2(\reals^d)$ boundedly to the Sobolev space $H^s(\reals^d)$
for $s=(d-1)/2$. $H^s$ embeds compactly into $L^2$ in any bounded region.
Therefore the weak $L^p$ convergence of $g_{\nu,\lambda}$ to $g^{(\lambda)}$
as $\nu\to\infty$ implies $L^{2}$ norm convergence
of $\tpar(g_{\nu,\lambda})$ to $\tpar(g^{(\lambda)})$ as $\nu\to\infty$,
for every fixed $\lambda$. Therefore
$\langle F_{\nu,\lambda},\tpar g_{\nu,\lambda}\rangle
\to \langle F^{(\lambda)},\tpar g^{(\lambda)}\rangle$ as $\nu\to\infty$,
for every fixed $\lambda<\infty$.

The conclusions of Lemma~\ref{lemma:summary}
guarantee that $g_{\nu,\lambda}\to g_\nu$
and $F_{\nu,\lambda}\to F_\nu$ 
in $L^p$ norm {\em uniformly} in $\nu$ as $\lambda\to\infty$. 
This uniform convergence, together with the convergence proved in the 
preceding paragraph, give \eqref{weaklimit}. 
\end{proof}

$L^p$ norm convergence of $g_\nu$ to $g$ will be proved below. 



\begin{proof}[Proof of part (iv) of Theorem~\ref{thm:main}]
Let $f$ be an arbitrary nonnegative extremizer satisfying $\norm{f}_{p }=1$.
For any $n$ there exist $\phi_n\in\scriptg_d$ and
a decomposition $\phi_n^* f = g_n+h_n$ where
$g_n$ satisfies the inequalities of Lemma~\ref{lemma:summary},
and $\norm{h_n}_p<2^{-n}$.
By passing to a subsequence, we may arrange that $g_n$ converges
in $L^p$ and almost everywhere to a limit $g$, which is an extremizer.
Since $\phi_n^*$ preserves the $L^p$ norm and $\norm{h_n}_p\to 0$,
$\phi_n^* f\to g$ and $(\phi_n^{-1})^* g\to f$ in $L^p$ norm.
It is an elementary consequence of the concrete description of
$\scriptg_d$ that this forces $\{\phi_n^{-1}\}$ 
to be a precompact family of diffeomorphisms of $\reals^d$.
By passing to a subsequence we conclude that 
$f = \phi^* g$ for some
$\phi\in\scriptg_d$.
Thus $f$ satisfies all of the conclusions stated in conclusion {\it (iv)}
of Theorem~\ref{thm:main}.
\end{proof}

\section{Loose ends}

\begin{proof}[Proof of the Euler-Lagrange identity \eqref{eulerlagrange}]
Let $f$ be 
any nonnegative extremizer satisfying $\norm{f}_p=1$.
Since $q=d+1$ is the exponent conjugate to $p=(d+1)/d$,
$\tpar^*$ has the same norm as $\tpar$,
as operators from $L^p$ to $L^q$. 
Then
\begin{multline*}
\apar^{d+1}\norm{f}_p^{d+1}
= \langle \tpar f,(\tpar f)^{d}\rangle
= \langle f,\tpar^*\big([\tpar f]^{d}\big)\rangle
\\
\le \norm{f}_p\norm{T^*(\tpar f)^d}_q
\le \apar \norm{f}_p\norm{(\tpar f)^d}_p
\\
= \apar \norm{f}_p\norm{\tpar f}_q^d
= \apar^{d+1}\norm{f}_p^{d+1}.
\end{multline*}
The first inequality is an application of H\"older's inequality.
Since the extreme left-- and right-hand sides are equal,
both inequalities in this chain must be equalities.
Equality in H\"older's inequality forces $\tpar^*\big([\tpar f]^{d})$
to agree almost everywhere with some constant multiple of $f^{1/d}$.
The chain of equalities then forces this constant to be $\apar^{d+1}$.
\end{proof}

\begin{lemma} \label{lemma:nonzeroae}
If $f$ is
a nonnegative extremizer for \eqref{eq:paraboloid},
then 
$f(x)>0$ for almost every $x\in\reals^d$.
\end{lemma}

\begin{proof}
Let $f$ be a nonnegative extremizer satisfying $\norm{f}_p=1$.
Consider initially the set $G$ consisting of all points
expressible in the form $z+(s',|s'|^2)-(t',|t'|^2)$
such that $z$ 
is a Lebesgue point of $\{y: f(y)>0\}$
and $s'\ne t'\in\reals^{d-1}$. 
We claim that $f(x)>0$ for almost every $x\in G$.
If not, choose some
subset $\tilde G\subset G$ satisfying $0<|\tilde G|<\infty$.
Consider the functions $f_\eps = f+\eps\chi_{\tilde G}$.
Then $\norm{f_\eps}_p^p= 1+O(\eps^p)= 1+o(\eps)$ as $\eps\to 0$.
On the other hand,
\[
\norm{\tpar f_\eps}_{d+1}^{d+1}
\ge \norm{\tpar f}_{d+1}^{d+1}
+ (d+1)\eps 
\int (\tpar f)^d\tpar\chi_{\tilde G}.\]
Now because $f,\chi_{\tilde G}$ are nonnegative
functions and $\tpar$ is defined by convolution
with a nonnegative measure,
$\int (\tpar f)^d\tpar\chi_{\tilde G}>0$ 
if and only if
$\int \tpar f\cdot\tpar\chi_{\tilde G}
= \langle \tpar^*\tpar f,\,\chi_{\tilde G}\rangle>0$.

$\tpar^*\circ\tpar$ is defined by convolution with
a measure which has a continuous, strictly positive
Radon-Nikodym derivative with respect to Lebesgue measure,
on the open set $\{(s',|s'|^2)-(t',|t'|^2): s'\ne t'\}$.
The condition that $\tilde G\subset G$ thus ensures that 
$\tpar^*\tpar f>0$ at every point of $\tilde G$.
Therefore 
$\langle \tpar^*\tpar f,\,\chi_{\tilde G}\rangle>0$.

Consequently $\norm{\tpar f_\eps}_{d+1}\ge \apar+c\eps$
for some $c>0$ for all sufficiently small $\eps>0$,
and therefore $\norm{\tpar f_\eps}_q/\norm{f_\eps}_p
\ge \apar+c'\eps$ for small positive $\eps$ for 
all $c'<c$, contradicting the extremality of $f$.
Thus any extremizer must be positive almost everywhere on some open set.


This additional information can be fed back into the above argument,
which then demonstrates that $f>0$ almost everywhere
at every point $z+(s',|s'|^2)-(t',|t'|^2)$ where
$z$ varies over a subset of full measure of some ball.
One more iteration establishes the conclusion.
\end{proof}

\begin{corollary}
Let $f$ be a nonnegative extremizer for \eqref{eq:paraboloid}.
Then for any compact set $K\subset\reals^d$
there exists $c>0$
such that $f(x)\ge c$ for almost every $x\in K$.
\end{corollary}

This follows from the Euler-Lagrange equation by reasoning
already used in the proof of Lemma~\ref{lemma:nonzeroae},
since $\tpar^*\circ\tpar$ is expressed by convolution with
a nonnegative measure $\mu$ with the property that any
point is expressible as a finite sum of elements of
an open set on which $\mu$ has a continuous, strictly
positive Radon-Nikodym derivative. 

It remains to be proved that any extremizing sequence has
a subsequence which converges
in $L^p$ norm, rather
than the weak convergence proved above.

\begin{lemma} \label{lemma:strongconvergence}
Let $\{f_\nu\}$ be an extremizing sequence for
inequality \eqref{eq:paraboloid};
thus $f_\nu\ge 0$,
$\norm{f_\nu}_p=1$,
and $\norm{\tpar f_\nu}_q\to \apar$.
Then there exist a sequence of symmetries
$\{\phi_\nu\}\subset\scriptg_d$,
an extremal $F$ for \eqref{eq:paraboloid},
and a subsequence $\{f_{\nu_k}\}$
such that
$\phi_{\nu_k}^*f_\nu\to F$
in $L^{(d+1)/d}(\reals^d)$ norm.
\end{lemma}

\begin{proof}
After passing to a subsequence, we may
choose $\phi_\nu$ so that $\phi_\nu^*f_\nu=F_\nu+h_\nu$,
so that
$\norm{h_\nu}_p\to 0$,
$\norm{F_\nu}^p\to 1$,
$F_\nu^p\rightharpoonup F^p$
for some nonnegative extremal $F$ satisfying $\norm{F}_p=1$,
and $F_\nu,F$ satisfy the higher integrability
and spatial decay bounds provided by Lemma~\ref{lemma:summary}.
Moreover, there exists $H\ge 0$ satisfying
$\norm{H}_p=1$ 
and $\tpar^* H = \apar F^{p-1}$ almost everywhere.

As in the proof of \eqref{weaklimit}
it follows from
the {\em a priori} bounds 
and Rellich's lemma that after passing to a subsequence,
$\langle F_\nu,\tpar^* H\rangle \to \langle F,\tpar^* H\rangle
= \apar \langle F,F^{p-1}\rangle = \apar$.
Therefore
$\langle F_\nu,F^{p-1}\rangle \to 1 = \langle F,F^{p-1}\rangle$.

For any $\nu$ and any small $\delta>0$
denote
$E_{\delta,\nu} = \{x: F_\nu^p(x) \le(1-\delta)F^p(x)\}$.
Recall that $q=d+1$ is the exponent dual to $p=(d+1)/d$.
Then
\begin{multline*}
\int F_\nu F^{p-1}
\le \int p^{-1}F_\nu^p + q^{-1} (F^{p-1})^q
= \int p^{-1}F_\nu^p + q^{-1} F^p
\\
\le p^{-1}\int F^p
+ q^{-1}\int F^p
-\delta p^{-1} \int_{E_{\delta,\nu}} F^p
= 
1-\delta p^{-1} \int_{E_{\delta,\nu}} F^p.
\end{multline*}
Therefore
\begin{equation*} 
\int_{E_{\delta,\nu}} F^p\to 0 \text{ as } \nu\to\infty
\end{equation*}

Because $F>0$ almost everywhere, 
this implies that for any compact subset $K\subset\reals^d$,
$|E_{\delta,\nu}\cap K|\to 0$
as $\nu\to\infty$.
Together with the higher integrability and spatial decay
bounds of Lemma~\ref{lemma:summary}, this implies that
$\int_{E_{\delta,\nu}} F_\nu^p\to 0$ also.
We may choose $\delta=\delta(\nu)$ to tend to 
zero as $\nu\to\infty$, yet still have 
\[\int_{E_{\delta(\nu),\nu}}(F^p+F_\nu^p)\to 0.\]

Now
\[
\int_{F_\nu^p(x)\ge (1-\delta)F^p(x)}
(F^p-F_\nu^p)
\le
\delta\int F^p\le\delta.
\]
Therefore by splitting the integral into the two regions 
$F_\nu^p\le (1-\delta(\nu))F^p$ and
$(1-\delta(\nu))F^p\le F_\nu^p\le F^p$
we conclude that
\begin{equation} \label{onesided}
\int_{F_\nu\le F} (F^p-F_\nu^p)\to 0
\text{ for each } \delta>0.
\end{equation}
Since $\int F_\nu^p=\int F^p$,
\eqref{onesided} forces $\int |F^p-F_\nu^p|\to 0$.
\end{proof}

One final conclusion of Theorem~\ref{thm:main} remains to be 
established, the existence of a decomposition
$g=g^\sharp+g^\flat$ where $\norm{g^\sharp}_{C^1}\le R$
and $\norm{g^\flat}_p\le\eta(R)$, where $\eta\to 0$ as $R\to\infty$.
Equivalently, if $\zeta\in C^\infty(\reals^d)$ is compactly supported
and satisfies $\zeta(0)=1$, and if $M_\rho$ is the Fourier
multiplier operator with symbol $1-\zeta(\xi/\rho)$
then $\norm{M_\rho g}_p\le\tilde\eta(\rho)$
where $\tilde\eta(\rho)\to 0$ as $\rho\to\infty$. 
Here $g$ is a $(1-\delta)$--quasiextremal which satisfies
\eqref{concl:higherintegrability} and \eqref{concl:decay},
and $\delta$ is permitted to depend on $R$ and/or $\rho$ and can be 
chosen to be as small as may be desired.

\begin{proof}[Proof of \eqref{concl:smoothness}]
Fix $F$ satisfying $\norm{F}_p=1$ such that
$\langle g,\tpar^*F\rangle\ge (1-\delta)\apar$.
Let $\eps>0$.
By the results proved above,
if $\rho$ is sufficiently large then
$\tpar^*F$ may be decomposed as
$\tpar^* F = H^\sharp+H^\flat$
where
$\widehat{H^\sharp}(\xi)=0$ for all $|\xi|\ge\rho/2$
and
$\norm{H^\flat}_q<\eps$.

Split $g=g^\sharp+g^\flat$
where $\widehat{g^\sharp}$ is supported
where $|\xi|<2\rho$,
$\widehat{g^\flat}$ is supported
where $|\xi|>\rho$,
and $\norm{g^\sharp}_p+\norm{g^\flat}_p
\le C<\infty$; this may be done with $C$
a constant depending only on $d,p$.
To do this, set $g^\flat = M_\rho g$
where $M_\rho$ has symbol $1-\zeta(\xi/\rho)$,
with the auxiliary function $\zeta\in C^\infty_0$
chosen to satisfy $\zeta(\xi)\equiv 1$ whenever
$|\xi|\le 2$. Then $g^\flat$ will continue to
satisfy \eqref{concl:higherintegrability}
and \eqref{concl:decay}, uniformly in $\rho,g$
so long as $\rho\ge 1$.
 
We aim to prove that if $\delta,\eps$ are sufficiently
small and $\rho$ is sufficiently large,
then $\norm{g^\flat}_p$ is less than any preassigned quantity.

Consider $\tilde g = g^\sharp-g^\flat$.
Then 
\begin{align*}
\langle \tilde g,\tpar^*F\rangle
&= \langle g,\tpar^*F\rangle
-2\langle g^\flat,\tpar^*F\rangle
\\
&= \langle g,\tpar^*F\rangle
-2\langle g^\flat,H^\flat \rangle
\\
&\ge (1-\delta)\apar
-2\norm{g^\flat}_p\norm{H^\flat}_q
\\
&\ge (1-\delta)\apar
-2C\eps.
\end{align*}
Thus
\[\norm{\tilde g}_p\ge (1-\delta) - 2C\eps\apar^{-1}.\]
The same reasoning gives
\begin{equation}\label{gsharplarge}
\norm{g^\sharp}_p\ge (1-\delta) - C\eps\apar^{-1}.
\end{equation}

Since $1<p<2$, there exists $c_p<\infty$ such that
for any $s,t\in\reals$,
\[
\tfrac12 |s+t|^p
+\tfrac12 |s-t|^p
\ge |s|^p + c_p\min(|s|^{p-2}|t|^2,|t|^p).
\]
Therefore
\begin{equation*}
\tfrac12 |g(x)|^p
+\tfrac12 |\tilde g(x)|^p
\ge
|g^\sharp(x)|^p + c_p\min(|g^\sharp(x)|^{p-2}|g^\flat(x)|^2,|g^\flat(x)|^p)
\end{equation*}
for all points $x\in\reals^d$.
Therefore
\begin{equation*}
1-O(\eps+\delta) \ge 
\int |g^\sharp|^p
+ c_p \int \min\big(|g^\sharp|^{p-2}|g^\flat|^2,|g^\flat|^p\big).
\end{equation*}
Let $\tilde\eps>0$ be another small quantity.
Then
\begin{equation*}
1-O(\eps+\delta) \ge 
\norm{g^\sharp}_p^p
+ c\tilde\eps^{2-p}\int_{|g^\flat(x)|>\tilde\eps|g^\sharp(x)|} |g^\flat(x)|^p\,dx
\end{equation*}
and consequently by \eqref{gsharplarge},
\begin{equation*}
 \tilde\eps^{2-p}
\int_{|g^\flat(x)|>\tilde\eps|g^\sharp(x)|} |g^\flat(x)|^p\,dx
= O(\eps+\delta)
\end{equation*}
and hence
\begin{equation*}
\norm{g^\flat}_p^p
\le 
\tilde\eps^p\norm{g^\sharp}_p^p
+ C\tilde\eps^{p-2}(\eps+\delta)
\le C\tilde\eps^p
+ C\tilde\eps^{p-2}(\eps+\delta).
\end{equation*}
By choosing first $\eps,\delta$ to be small,
then $\tilde\eps= (\eps+\delta)^{1/2(2-p)}$,
we arrive at the desired bound.
\end{proof}

It is natural to ask whether the set of all extremizers
of \eqref{eq:paraboloid} satisfying $\norm{f}_p=1$ 
is unique modulo the action of $\scriptg_d$. We can offer only the following
weak substitute.

\begin{corollary}
Let $\{F_\nu\}$ be any sequence of 
extremizers of \eqref{eq:paraboloid}
which satisfy $\norm{F_\nu}_p=1$. 
Then there exist a subsequence $\{F_{\nu_k}\}$
and a sequence of symmetries $\phi_k\in \scriptg_d$
such that the sequence
$\{\phi_k^*F_{\nu_k}\}$
is convergent in $L^p(\reals^d)$ norm.
\end{corollary}

This follows directly from Lemma~\ref{lemma:strongconvergence}.

\section{On the symmetry groups} \label{onsymmetry}

To list all elements of $\scriptg_{d,d}$, we
employ coordinates $(x,y) = (x',x_d;y',y_d)$
with $x',y'\in\reals^{d-1}$ and $x_d,y_d\in\reals^1$.
$\scriptg_{d,d}$ is a subset of the set of all 
elements $\Phi$ of $\diff(\reals^d)\times\diff(\reals^d)$
of the form
\begin{equation}\label{diffeoform}
(x',x_d;y',y_d)
\mapsto \Big(Lx'+u, tx_d+a+x'\cdot v+Q(x');
\tilde L y'+\tilde u, ty_d+\tilde a + y'\cdot \tilde v + \tilde Q(y')\Big),
\end{equation}
where $L,\tilde L$ are linear endomorphisms of $\reals^{d-1}$,
$v,\tilde v\in\reals^{d-1}$,
$0\ne t\in\reals$,
$a,\tilde a\in\reals$,
and
$Q,\tilde Q:\reals^{d-1}\to\reals$ are
homogeneous quadratic polynomials.
The proof of the next lemma is a straightforward
verification, left to the reader.

\begin{lemma}
A mapping $\Phi$ of the form \eqref{diffeoform}
belongs to $\scriptg_{d,d}$
if and only if all of the following equations hold:
\begin{gather*}
-a-\tilde a +2u\cdot\tilde u-|u|^2-|\tilde u|^2 =0
\\
\tilde L^* L = tI
\text{ for some } t\in\reals\setminus\{0\}
\\
v + 2L^*(\tilde u-u)=0
\\
\tilde v + 2\tilde L^*(\tilde u-u)=0
\\
Q(x') = |Lx'|^2-t|x'|^2
\\
\tilde Q(y') = |\tilde Ly'|^2-t|y'|^2.
\end{gather*} 

Given any $L,u,t,a,v$ there exist 
$Q,\tilde Q,\tilde L,\tilde u,\tilde a,\tilde v$
satisfying all these equations, and these quantities
are uniquely determined by $L,u,t,a,v$.
\end{lemma}

\begin{definition}
A collection of $d$ points $x_j \in\reals^{d}$
is said to lie in general position if, writing 
$x_j=\big(x_{j,1},\cdots, x_{j,d}\big)$,
the $d\times d$ matrix
\[
\begin{pmatrix}
x_{1,1} & x_{1,2}& \cdots& x_{1,d-1} & 1
\\
x_{2,1} & x_{2,2}& \cdots& x_{2,d-1} & 1
\\
\vdots & \vdots & \vdots & \vdots & \vdots
\\
x_{d,1} & x_{d,2}& \cdots& x_{d,d-1} & 1
\end{pmatrix}
\]
is nonsingular.
\end{definition}
This definition is coordinate-dependent. The $d$-th coordinates
of the points $x_j$ do not enter into consideration.


\begin{lemma}
The action of $\scriptg_d$ on $\reals^d$
is generically $d$-fold transitive.
That is,
for any two sets $\{x_j: 1\le j\le d\}$ 
and
$\{y_j: 1\le j\le d\}$  
of $d$ points in general position in $\reals^d$,
there exists
$\phi\in\scriptg_d$ satisfying 
\[\phi(x_j)=y_j \text{ for all $1\le j\le d$.}\]
\end{lemma}

\begin{proof}
There exists  a unique affine endomorphism $x'\mapsto Lx'+u$
of $\reals^{d-1}$ which maps $x'_j$ to $y'_j$ for all $1\le j\le d$.
Fix any $0\ne t\in\reals$.
Define $Q(y') = |L(y')|^2-t|y'|^2$ for $x'\in\reals^{d-1}$.
Let $a\in\reals$ and $v\in\reals^{d-1}$ be parameters to be specified.
Consider the diffeomorphism $\phi\in\scriptg_d$ defined by
$\phi(x',x_d) = (Lx'+u,tx_d+a + v\cdot x' + Q(x'))$.

We merely need to choose these parameters to satisfy the equations
\[tx_{j,d}+a+v\cdot x'_j + Q(x'_j) = y_{j,d}
\text{ for all $1\le j\le d$,} \] 
which is to say,
\[
v\cdot x'_j +a = 
y_{j,d}
-|L(x'_j)|^2 
-t \big(x_{j,d} -|x'_j|^2\big). 
\]
Since $t$ has already been specified, this is a system of $d$  inhomogeneous linear
equations for $(v,a)\in\reals^{(d-1)+1}$.
The assumption that
$\{x_j: 1\le j\le d\}$ is in general position means precisely that 
this system has full rank $d$.
Thus for any $t$, there exists a (unique) solution.
\end{proof}


The proof also demonstrates that $\phi$ is uniquely determined, up to the
``vertical'' scaling factor $t$.

\section{Affine surface measure} \label{section:affine}

Our definition of $\tpar$ uses the measure $dt$, rather than the surface measure
on the paraboloid induced from its inclusion into $\reals^d$. This 
measure $dt$ is entirely natural from a geometric viewpoint, which we now explain.

First, we recall the definition and properties of affine {\em arclength} measure.
This is a measure on a subinterval
$I\subset\reals$, associated to any sufficiently
smooth mapping $\gamma:I\to\reals^d$,
as follows:
Define 
\[
L_\gamma(t) = \big|\det\big(\gamma'(t),\gamma^{(2)}(t),\cdots,\gamma^{(d)}(t)\big)\big|
\]
where $\gamma^{(k)}(t) = d^k\gamma(t)/dt^k$.
Then the measure $\sigma_\gamma(I)$ is defined  to be 
\begin{equation}\sigma_\gamma(I) = \int_I |L_\gamma(t)|^{2/d(d+1)}\,dt.\end{equation}
See \cite{gugg} for some discussion.

Affine arclength measure $\sigma_\gamma$ enjoys two natural invariances,
which provide its {\em raison d'\^{e}tre}:
\newline
(i) For any $A\in Gl(\reals,d)$, $\sigma_{A\circ\gamma}(I) 
= |\det(A)|^{2/d(d+1)}\sigma_\gamma(I)$.
\newline
(ii) For any injective $C^d$ mapping $\phi:I\to\reals$,
$\sigma_{\gamma\circ\phi}(I) = \sigma_\gamma(\phi(I))$.
\newline
Both identities are easily verified.

There is a natural analogue in the codimension one case,
which we call affine surface measure.
Let $F$ be any $C^2$ mapping from an open subset of $\reals^{d-1}$
to $\reals^d$.
Write $F(t) = (F_1(t), \cdots, F_d(t))$.
For each pair of indices $i,j\in\{1,2,\cdots,d-1\}$, form
\[
F_{i,j}(t) = 
\det
\begin{pmatrix}
\frac{\partial F_1(t)}{\partial t_1}
& \cdots &
\frac{\partial F_1(t)}{\partial t_{d-1}}
&
\frac{\partial^2 F_1(t)}{\partial t_i \partial t_j}
\\
\frac{\partial F_2(t)}{\partial t_1}
& \cdots &
\frac{\partial F_2(t)}{\partial t_{d-1}}
&
\frac{\partial^2 F_2(t)}{\partial t_i \partial t_j}
\\
\vdots & \cdots & \vdots & \vdots \\
\frac{\partial F_d(t)}{\partial t_1}
& \cdots &
\frac{\partial F_d(t)}{\partial t_{d-1}}
&
\frac{\partial^2 F_d(t)}{\partial t_i \partial t_j}
\end{pmatrix}.
\]
From these quantities $F_{i,j}$ we form a $(d-1)\times(d-1)$ matrix,
and consider its determinant
\[
\scriptl_F(t) = 
\det
\begin{pmatrix}
F_{i,j}(t)
\end{pmatrix}_{i,j=1}^{d-1}.
\]
For any Borel set $U\subset\reals^{d-1}$ contained
in the domain of $F$, define the affine surface measure of $U$ to be 
\[
\sigma_F(U) = \int_U |\scriptl_F(t)|^{1/(d+1)}\,dt.
\]

Affine surface measure enjoys corresponding invariances.
\begin{proposition} \label{prop:affinesurface}
Let $A\in Gl(\reals,d)$ and let $\phi:V\to U$ be a $C^2$
diffeomorphism of an open set $V\subset\reals^{d-1}$
with a subset $U$ of the domain of $F$.
Then 
\begin{align} \label{affinesurface1}
\sigma_{A\circ F}(U) &= |\det(A)|^{(d-1)/(d+1)}\sigma_F(U)
\\
\sigma_{F\circ\phi}(V) &= \sigma_F(\phi(V)).
\label{affinesurface2}
\end{align}
\end{proposition}

In the case where $F(t) = (t_1,\cdots,t_{d-1},|t|^2)$,
affine surface measure is a constant multiple of Lebesgue
measure on $\reals^{d-1}$, which is the measure used in
our definition of $\tpar$.
Thus our use of this measure, rather than of surface
measure on the paraboloid, has a natural geometric context.

\begin{proof}[Proof of Proposition~\ref{prop:affinesurface}]
Property \eqref{affinesurface1} follows directly from the definition; 
passing from $F$ to $A\circ F$ multiplies each
$F_{i,j}$ by $\det(A)$.
To derive \eqref{affinesurface2}, 
note first that if $G = F\circ\phi$, then for any indices $i,j$,
the determinant $G_{i,j}(t)$ depends only on the Jacobian matrix $D\phi(t)$
and on the first and second partial derivatives of $F$ at $\phi(t)$. 
Indeed, the chain rule produces an undesired contribution involving first derivatives of $F$ and
quadratic expressions involving $D\phi$ in the rightmost column,
but this contribution is a linear combination of the first $d-1$ columns and
hence can be eliminated via row operations.

Thus matters reduce to the case where $\phi$ is linear,
and it is no loss of generality to suppose that $t=0$ and $F(0)=0$.
We claim then that 
\begin{equation} \label{affinesurface3}
|\scriptl_{F\circ\phi}(0)| = |\det(\phi)|^{d+1}|\scriptl_F(0)|. 
\end{equation}
Taking into account both the formula $d\tau = |\det(D\phi)|\,dt$, where $\tau=\phi(t)$,
and the exponent $1/(d+1)$ in the definition of $\sigma_{F\circ\phi}$,
would then yield \eqref{affinesurface2}.

To derive \eqref{affinesurface3},
the invariance \eqref{affinesurface1} already established can be used
to reduce matters to the case where $F$ takes the form
$F(t) = (t_1,t_2,\cdots,t_{d-1},f(t))$
where $Df(0)=0$.
Then 
\[(F\circ\phi)_{i,j}(0)
= \det(\phi)
\cdot
\sum_{k,l}\partial^2_{i,j} f(0) \phi_{k,i}\phi_{l,j}
\]
where $\partial^2_{i,j} f$ denotes the second partial derivatives of $f$,
and $\phi = \begin{pmatrix} \phi_{k,i}\end{pmatrix}_{k,i=1}^{d-1}$.
Thus
\[
\scriptl_{F\circ\phi}(0) = \det(\phi)^{d-1} \cdot
\det\begin{pmatrix}
\sum_{k,l}\partial^2_{i,j} f(0) \phi_{k,i}\phi_{l,j}
\end{pmatrix}_{i,j=1}^{d-1}.
\]
This last matrix is simply the composition $\phi^*\circ\partial^2 f(0)\circ\phi$
where $\phi^*$ denotes the transpose of $\phi$.
Therefore
\[
\scriptl_{F\circ\phi}(0) 
= \det(\phi)^{d-1} \cdot
\det(\phi)^2\det(\partial^2 f(0))
=
\det(\phi)^{d+1} \scriptl_F(0),
\]
which is \eqref{affinesurface3}.
\end{proof}

\medskip
Inequalities for operators defined by convolution with (the push-forward onto a curve of)
affine arclength are studied in \cite{dlw} and \cite{betsyaffine}.
See the former paper for references to other related works.
One may ask whether there are analogues, for hypersurfaces equipped with
affine surface measure, of the results of \cite{dlw}, \cite{betsyaffine}
for curves equipped with affine arclength measure.

\end{document}